\documentclass[a4paper, 12pt, twoside, notitlepage,reqno]{amsart}
\usepackage{amsmath,amscd,amsfonts}
\usepackage{amssymb}
\usepackage{amsthm}
\usepackage{enumerate}
\usepackage{fullpage}

\usepackage{mathrsfs}  
\usepackage{enumitem}
\usepackage{orcidlink}
\usepackage{cleveref}

\usepackage[utf8]{inputenc}
\usepackage{csquotes}
\usepackage{verbatim}
\usepackage[T1]{fontenc}
\usepackage{textcomp}  
\usepackage{mathtools,amssymb}
\allowdisplaybreaks

\usepackage[mathscr]{eucal} 
\usepackage{multirow}

\usepackage{tikz}
\usetikzlibrary{arrows.meta, positioning, calc, fadings, shapes.geometric, shapes.misc, intersections}
\tikzset{cross/.style={cross out, draw=black, minimum size=2*(#1-\pgflinewidth), inner sep=0pt, outer sep=0pt},
cross/.default={1pt}} 

\usepackage{tikz-cd}

\newtheorem{theorem}{Theorem}[section]
\newtheorem{lemma}[theorem]{Lemma}
\newtheorem{corollary}[theorem]{Corollary}

\theoremstyle{definition}
\newtheorem{definition}[theorem]{Definition}

\newtheorem{assumption}[theorem]{Assumption}
\newtheorem{proposition}[theorem]{Proposition}
\newtheorem{notation}[theorem]{Notation}
\theoremstyle{remark}
\newtheorem{remark}[theorem]{Remark}

\numberwithin{equation}{section}

\newcommand{\R}{\mathbb{R}}
\newcommand{\ro}{r_{\mathrm{out}}}

\newcommand{\etan}{{\epsilon_\tau}}

\newcommand{\mcC}{\mathcal{C}}
\newcommand{\mcG}{\mathcal{G}}

\newcommand{\mcS}{\mathcal{S}}
\newcommand{\mcR}{\mathcal{R}}
\newcommand{\mcT}{\mathcal{T}}
\newcommand{\mcV}{\mathcal{V}}
\newcommand{\mcM}{\mathcal{M}}
\newcommand{\utan}{u_{\tan}}
\DeclareMathOperator{\sgn}{sgn}

\mathtoolsset{showonlyrefs}

\begin{document}

\title[Singular Stability Estimates for PGET]{Singular Stability Estimates for Passive Gamma Emission Tomography of Spent Nuclear Fuel}

\author[T. Heikkil\"a]{Tommi Heikkil\"a\,\orcidlink{0000-0001-5505-8136}}
\address{Computational Engineering, School of Engineering Sciences,
Lappeenranta-Lahti University of Technology LUT, Lappeenranta, Finland}
\email{tommi.heikkila@lut.fi}

\author[T. Helin]{Tapio Helin\,\orcidlink{0000-0002-5788-0008}}
\address{Computational Engineering, School of Engineering Sciences,
Lappeenranta-Lahti University of Technology LUT, Lappeenranta, Finland}
\email{tapio.helin@lut.fi}

\author[S. Holman]{Sean Holman\,\orcidlink{0000-0001-8050-2585}}
\address{Department of Mathematics, University of Manchester, Manchester, United Kingdom}
\email{sean.holman@manchester.ac.uk}

\author[S. R. Jathar]{Shubham R. Jathar\,\orcidlink{0000-0002-9109-7543}}
\address{Computational Engineering, School of Engineering Sciences,
Lappeenranta-Lahti University of Technology LUT, Lappeenranta, Finland}
\email{shubham.jathar7@gmail.com, shubham.jathar@lut.fi}

\author[J. Railo]{Jesse Railo\,\orcidlink{0000-0001-9226-4190}}
\address{Computational Engineering, School of Engineering Sciences,
Lappeenranta-Lahti University of Technology LUT, Lappeenranta, Finland}
\email{jesse.railo@lut.fi}



\begin{abstract}
We study the joint recovery of attenuation and source coefficients from the attenuated X-ray transform, motivated by passive gamma
emission tomography of spent nuclear fuel. Both coefficients
are constant on known, disjoint disks of equal radius
representing fuel rods. The surrounding moderator has known attenuation and zero source. We use tangential singularities of the data to give a sufficient condition on the geometry and coefficients for unique joint reconstruction from full data. For rods arranged on parallel rows separated by more than one rod diameter, with source coefficients uniformly bounded away from zero, we derive explicit reconstruction formulas and stability estimates for both coefficients. The estimates use a singular seminorm determined by two functionals per rod, computed from angular data near selected tangencies. Numerical experiments examine the approximation of these functionals from discretized data.
\end{abstract}

\maketitle
\section{Introduction}\label{sec:Intro}
Passive gamma emission tomography (PGET) is a non-invasive technique for verifying the internal configuration of spent nuclear fuel assemblies \cite{Bubba:2020:pget}. Gamma rays emitted by the fuel rods are attenuated as they pass through the rods and the moderator before reaching a row of collimated gamma-ray detectors. Mathematically, these measurements are modeled by the \emph{attenuated X-ray transform} (AtRT) $R_af$, where $f$ describes the emission and $a$ the attenuation in a domain $\Omega$.

In standard emission tomography, the attenuation is often treated as known, at least approximately, and only the source is recovered. Reconstruction methods include iterative approaches \cite{shepp1982maximum,hudson1994accelerated} and, for the attenuated X-ray transform, Novikov's inversion formula \cite{Novikov:2002}. In PGET, both $a$ and $f$ are unknown and both are of interest. Their joint recovery from the AtRT is also known as the SPECT identification problem \cite{bronnikov1999numerical} and has been studied since the late 1970s \cite{censor1979new,natterer1983computerized}. The problem is not uniquely solvable in general, so unique recovery requires additional structure. 

We impose geometric and material assumptions motivated by the design of nuclear fuel assemblies. The domain is a fixed disk containing finitely many disjoint smaller disks of equal radius at known locations. The smaller disks represent cross-sections of fuel rods. We model $a$ and $f$ as constant on each rod, with known attenuation and zero source in the moderator. Holman and Richardson \cite{Holman:Richardson:2023} analyzed the singularities associated with tangential rays and recovered attenuation jumps under a non-cancellation condition.

Our contribution has two parts. First, we give a condition involving a common tangent and a second emitting rod that ensures non-cancellation on an open family of tangencies of a rod with a nonzero attenuation jump (Lemma~\ref{lm:noncancellation}). This gives a sufficient condition for recovery from full data (Theorem~\ref{tm:recovery}). Second, for rods on sufficiently separated parallel rows with strictly positive rod sources, we select tangencies at which non-cancellation holds uniformly over the parameter space. These tangencies give explicit reconstruction formulas and stability estimates for both coefficients in a specific singular data seminorm (Theorem~\ref{tm:stability}). 
The analysis does not include noise or scattering. Stability remains open when the assembly contains a non-emitting rod, such as fresh nuclear fuel or a special absorber rod.

\subsection{Related literature}
When $a=0$, the AtRT is the classical Radon transform. For known constant attenuation in a known convex domain, removing the factor determined by the ray's exit point gives the exponential Radon transform \cite{Natterer:2001:book}, for which explicit inversion
formulas are available \cite{Tretiak:1980}. Hertle \cite{Hertle:1988} studied the identification of an unknown constant attenuation parameter when the source is also unknown. Solmon \cite{Solmon:1995} studied the corresponding identification problem, including the role of radial symmetry. Novikov established an explicit inversion formula \cite{Novikov:2002} and a range characterization \cite{Novikov:2002:range} for the two-dimensional attenuated X-ray transform. Natterer \cite{Natterer:2001:inversion} derived a closely related exact inversion formula using a different approach. Boman and Str\"omberg \cite{Boman:Stromberg:2004}
gave an alternative proof of Novikov's inversion formula and extended it to a broader class of weighted Radon transforms. Kunyansky \cite{Kunyansky:2001} developed a noniterative
SPECT reconstruction algorithm based on Novikov's inversion formula for known, spatially varying attenuation. An earlier approach used the theory of $A$-analytic functions \cite{Arbuzov:Bukhegim:Kazantsev:1998}. 
Surveys include \cite{Finch:2003,Novikov:2019,Zhou:2024:survey,GIP2D}.

Dicken \cite{dicken1999new} studied joint recovery using the stability and mapping properties of the AtRT in $L^p$ and Sobolev spaces, together with nonlinear Tikhonov regularization. Stefanov \cite{Stefanov:2014} studied uniqueness, stability, and the linearized problem, and also demonstrated non-uniqueness when both coefficients are radial. Holman and Richardson \cite{Holman:Richardson:2020} established unique recovery under a piecewise constant assumption on the attenuation and with a sufficiently regular source. Their subsequent work \cite{Holman:Richardson:2023} treated piecewise smooth sources and real-analytic interfaces under a non-cancellation condition.

PGET has been developed for nuclear safeguards and the verification of spent nuclear fuel \cite{white2018application,Virta:2024}, and has been approved by the International Atomic Energy Agency (IAEA) since 2017 \cite{white2018application}. A reconstruction method used in practice combines an iterative Levenberg--Marquardt trust-region algorithm with physical constraints and variational regularization \cite{Bubba:2020:pget,virta2022improved,Virta:2024}. Other methods have been developed, for example in \cite{svard2015applicability,white2017spect,fang2021quantitative}. The fixed rod geometry of PGET motivates our analysis of uniqueness and stability.

\subsection{Organization of the article}
This article addresses two questions in the PGET disk model.
\begin{enumerate}
    \item When can non-cancellation be established from the geometry and the coefficient assumptions?
    \item How stable are the recovered attenuation and source with respect to perturbations of the data?
\end{enumerate}
Section~\ref{sec:uniqueness} addresses the first question, and Section~\ref{sec:stability} addresses the second. We present related numerical experiments in Section~\ref{sec:numerics} to help illustrate some of the presented concepts and theoretical results. These codes are also available in \url{https://github.com/tommheik/PGET_stability}.

The following diagram summarizes the structure of the proofs of the two main results.
\begin{equation}\label{eq:proof_structure}
\begin{tikzcd}[column sep=small, row sep=normal, cells={nodes={font=\small}}, arrows={-{Straight Barb[length=1.2mm]}}]
& \text{Lemma~\ref{lm:u_tan_as_sum}} \arrow[d] & & \text{\cite{Novikov:2002}} \arrow[d] & \\
\begin{gathered}
    \text{Lemma~\ref{lm:d_q_ell_q}}\\
    \text{Remark~\ref{rm:chord_analytic}}
\end{gathered} \arrow[r] \arrow[ur]
& \text{Corollary~\ref{cor:u_tan_value}} \arrow[r] \arrow[ddr]
& \text{Lemma~\ref{lm:noncancellation}} \arrow[r]
& \text{Theorem~\ref{tm:recovery}}
& \text{Lemma~\ref{lm:window}} \arrow[d] \\
& & & \text{Proposition~\ref{prop:HR}} \arrow[u] \arrow[d]
& \text{Lemma~\ref{lm:side}} \arrow[l] \\
& & \text{Proposition~\ref{prop:existence_admissible_family}} \arrow[r]
& \text{Theorem~\ref{tm:stability}}
& \text{Remark~\ref{rm:seminorm}} \arrow[l] \\
& & & \text{Lemma~\ref{lm:L_bounds}} \arrow[u] &
\end{tikzcd}
\end{equation}
We prove Theorem~\ref{tm:recovery} in Section~\ref{sec:uniqueness}. The main step is Lemma~\ref{lm:noncancellation}, which combines the chord formulas of Lemma~\ref{lm:d_q_ell_q} with the tangential trace representation in Corollary~\ref{cor:u_tan_value} to show that the non-cancellation coefficient does not vanish identically along a family of simple tangencies near a common tangent. The singular identities of Proposition~\ref{prop:HR}, based on
\cite{Holman:Richardson:2023}, then determine the attenuation. Once the attenuation is known, Novikov's inversion formula~\cite{Novikov:2002} recovers the source.

The proof of Theorem~\ref{tm:stability} in Section~\ref{sec:stability} uses the admissible tangencies constructed in Proposition~\ref{prop:existence_admissible_family}. Under the row-separation assumption, this proposition provides, for every rod, a tangency at which the tangential trace vanishes. The positive lower bound on the rod sources therefore gives a uniform non-cancellation bound. Combining the identities of Proposition~\ref{prop:HR} with
the geometric bounds of Lemma~\ref{lm:L_bounds} and the seminorm estimates of Remark~\ref{rm:seminorm} yields the stability estimates. 
\section{Measurement geometry and mathematical model}
We introduce the forward model and the fixed geometry, then describe tangencies and backward-ray chords used in reconstruction. We also specify the coefficient space and the singular data functionals.
\subsection{The attenuated X-ray transform}
We work in the planar cross-section $\mathbb{R}^2$. For an angle $\omega\in [0,2\pi)$, we set
\begin{equation}\label{eq:theta(omega)_perp_def}
    \theta(\omega) = \begin{pmatrix}\cos\omega\\ \sin\omega\end{pmatrix}\in\mathbb{S}^1,
\qquad
\theta^\perp(\omega) = \partial_\omega \theta(\omega)
= \begin{pmatrix}-\sin\omega\\ \cos\omega\end{pmatrix}.
\end{equation}
Angles are understood modulo $2\pi$. Thus, the space of angle pairs is the torus $(\mathbb R/2\pi\mathbb Z)^2$, represented by $[0,2\pi)^2$.
Angular neighborhoods and one-sided limits are expressed using local real representatives. The directed line through $x \in \mathbb{R}^2$ with direction $\theta(\omega)$ is
\begin{equation}\label{eq:line_1} L(x,\theta(\omega)) := \{x+t\theta(\omega):t\in\mathbb R\}. \end{equation}
The attenuation $a$ and source $f$ are bounded nonnegative measurable functions that vanish outside the bounded domain $\Omega$ specified below. The \emph{beam transform} of $a$ is  
\begin{equation}\label{eq:beam}
    Da(x,\theta(\omega))=\int_0^\infty a(x+\rho\theta(\omega))d\rho.
\end{equation}
The beam transform measures the attenuation accumulated along the forward ray from $x$ in the direction $\theta(\omega)$. The \emph{attenuated X-ray transform} (AtRT) of $f$ with attenuation $a$ is 
\begin{equation}\label{eq:AtRT}
    R_a f(x,\theta(\omega))=\int_{-\infty}^\infty f(x+t\theta(\omega)) e^{-Da(x+t\theta(\omega),\theta(\omega)) } dt.
\end{equation}
A change of variable shows that
\[
    R_af(x+s\theta,\theta)=R_af(x,\theta),
    \qquad s\in\mathbb R.
\]
Thus, the data depend only on the oriented line, independently of the point chosen to represent it.

For a bounded convex Lipschitz domain $\Omega\subset \mathbb R^2$ with unit outward normal $\nu$, let 
\[\Gamma^-=\bigl\{(x,\theta)\in\partial\Omega\times\mathbb{S}^1 : \theta\cdot \nu(x)\le 0\bigr\}.\]
The transport equation 
\begin{equation}\label{eq:transport}
    \theta\cdot\nabla_x u(x,\theta) + a(x)u(x,\theta) = f(x),\qquad (x,\theta)\in\Omega\times\mathbb{S}^1, \qquad u|_{\Gamma^-} = 0,
\end{equation}
has a solution $u$ given by
\begin{equation}\label{eq:u_sol_1}
    u(x,\theta(\omega)) = \int_{-\infty}^0f(x+t\,\theta(\omega))\exp\Bigl(-\int_t^0 a(x+(t-\rho)\theta(\omega))d\rho \Bigr)dt,
\end{equation}
see \cite[Section~2]{Holman:Richardson:2023}. For every $x\in\R^2$ and $\omega\in[0,2\pi)$, a change
of variables in \eqref{eq:u_sol_1} gives the backward-ray
representation
\begin{equation}\label{eq:u_backward}
    u(x,\theta(\omega))
    =
    \int_0^\infty
        f(x-s\theta(\omega))
        \exp\left(
            -\int_0^s a(x-\sigma\theta(\omega))\,d\sigma
        \right)\,ds.
\end{equation}

For $y\in\mathbb R^2$, $\theta\in\mathbb S^1$, and $s\ge0$,
the beam transform satisfies
\[
    Da(y-s\theta,\theta)
    =\int_0^s a(y-\sigma\theta)\,d\sigma+Da(y,\theta).
\]
Consequently, setting $s=-t$ and using
\eqref{eq:u_backward} give
\begin{align*}
    \int_{-\infty}^0
       f(y+t\theta)e^{-Da(y+t\theta,\theta)}\,dt
    &=
    e^{-Da(y,\theta)}
    \int_0^\infty
       f(y-s\theta)
       \exp\left(
           -\int_0^s a(y-\sigma\theta)\,d\sigma
       \right)\,ds\\
    &=e^{-Da(y,\theta)}u(y,\theta).
\end{align*}
Splitting the integral defining $R_af$ at $t=0$ yields
\[
    R_af(y,\theta)
    =e^{-Da(y,\theta)}u(y,\theta)
     +\int_0^\infty
        f(y+t\theta)e^{-Da(y+t\theta,\theta)}\,dt.
\]
Let $x\in\partial\Omega$ be a point where the outward normal
$\nu(x)$ is defined, and suppose that
$\theta\cdot\nu(x)\ge0$. By convexity, the forward ray $\{x+t\theta:t>0\}$ does not enter $\Omega$.  The zero extension of $a$ gives $Da(x,\theta)=0$.
The forward source integral also vanishes. The data and the solution $u$ therefore satisfy
\[
    R_af(x,\theta)=u(x,\theta).
\]
\subsection{Geometric setting of PGET}\label{def:PGET-geom}

Fuel rods are enclosed in metal cladding and arranged in assemblies containing dozens or hundreds of rods. Assembly designs include rectangular and hexagonal arrangements, water channels of different sizes, and partial-length rods. For details, see \cite{Virta:2024}. For a specified assembly and axial slice, the rod positions and diameters are known. Gamma-ray measurements are acquired with collimated detectors located on a ring outside the assembly. Figure~\ref{fig:PGET_device} shows a PGET device.

\begin{figure}
    \centering
    \includegraphics[width=0.6\linewidth]{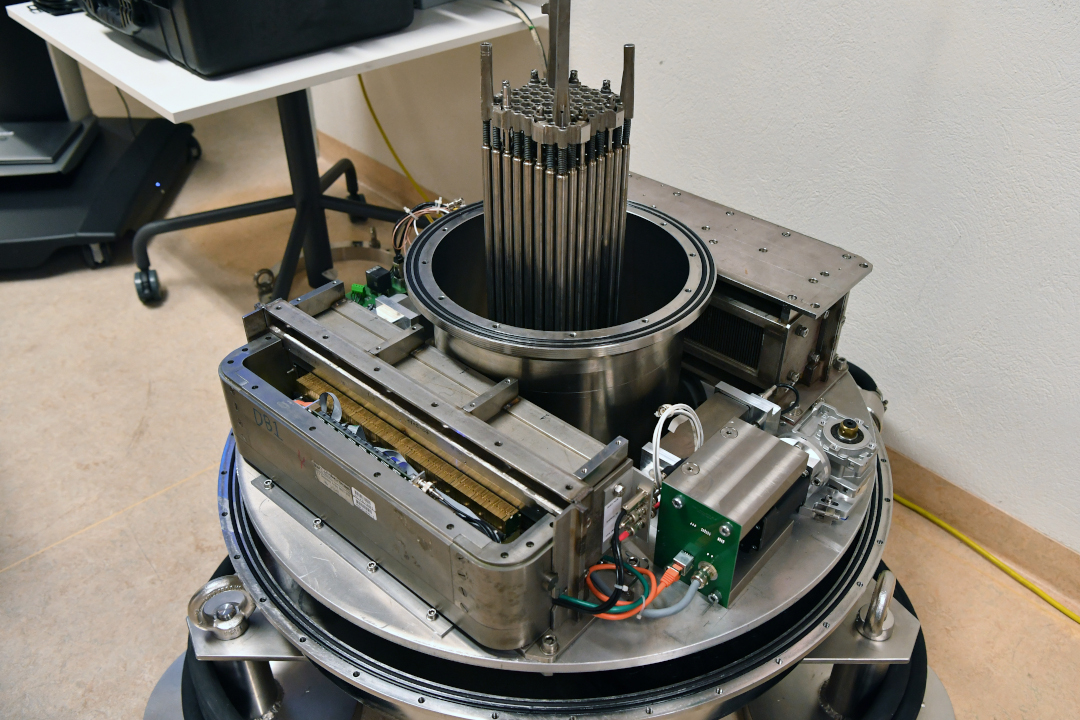}
    \caption{PGET device with its outer cover removed and a shortened dummy assembly. Measurements are collected around a horizontal slice of the assembly. Photograph by Dean Calma / IAEA.}
    \label{fig:PGET_device}
\end{figure}

We model the slice as disks of equal radius with constant coefficients, surrounded by a moderator with known attenuation and zero source. The model omits cladding, axial variation, scattering, and the finite collimator response. Known moderator-filled positions, including water channels and partial rods fully outside the field of view, belong to the moderator. Every other non-emitting rod remains a rod in the coefficient model. We also assume that the location of the assembly within the domain is known.

We now specify this geometry. Fix an outer radius $\ro>0$ and let $\Omega=B(0,\ro)$, with boundary
\begin{align}\label{eq:outer_boud_para}
    \partial\Omega=\{x(\psi): x(\psi)=\ro\bigl(\cos\psi,\sin\psi\bigr), \psi\in [0,2\pi)\}.
\end{align}
Fix the number of rods $K\ge 2$, their radius $r>0$ and centers $c_1,\dots,c_K$ such that the closed
disks
\[
    \overline{B_q}=\overline{B(c_q,r)}\subset\Omega,\qquad q=1,\dots,K,
\]
are pairwise disjoint. Equivalently, $|c_q-c_{q'}|>2r$ for $q\ne q'$ and $|c_q|+r<\ro$ for every $q$. The rod boundaries are $\partial B_q$, $q=1,\dots,K$, and we write $\Omega_0=\Omega\setminus\bigcup_{q=1}^K\overline{B_q}$
for the moderator region. The disks $B_1,\dots,B_K$ model the fuel or metal rods, and $\Omega_0$ models the surrounding water. We illustrate this in
Figure~\ref{fig:Detector-Fuelrod-setup}. These parameters define the geometry
\begin{equation}\label{eq:geometry}
    \mcG=\bigl(\ro,\;r,\;\{c_q\}_{q=1}^K\bigr),
\end{equation}
which remains fixed throughout the article. For $(\psi,\omega)\in[0,2\pi)^2$, we abbreviate
$L(\psi,\omega):=L(x(\psi),\theta(\omega))$.
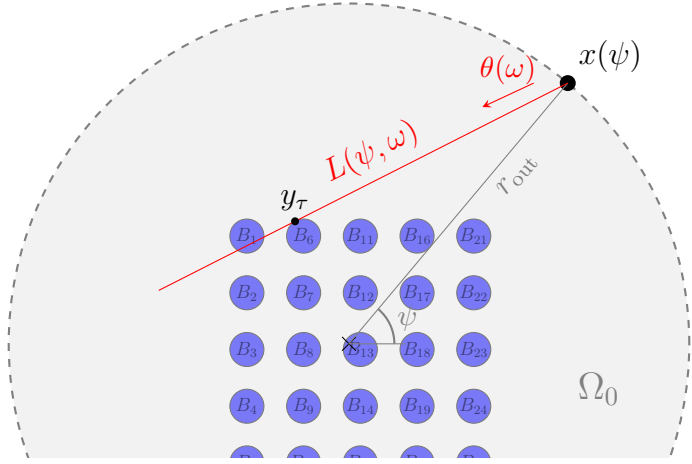
\begin{figure}
    \centering
    \begin{tikzpicture}
    \begin{scope}[scale=1.5]
        \path[clip] (-4,4) rectangle (4,-1); 


        \coordinate (xpsi) at (50:3cm); 
        \draw[thick, dashed, gray, fill=gray!20!white, fill opacity=0.5] (0,0) circle (3cm);
        \node[gray] at (2.2,-0.4) {\large $\Omega_0$};

        \coordinate (c) at (-0.9, -1.05);
        \foreach \i in {0, 1, ..., 4}{
            \foreach \j in {0, 1, ..., 4}{
                \pgfmathparse{int(5 + 5*\i - \j)} 
                \draw[black, fill=blue, opacity=0.5] ($(c) + 0.5*(\i,\j)$) circle (1.5mm) node[scale=0.6] {$B_{\pgfmathresult}$};
            }
        }
    
        \draw[gray] (0,0) node[cross=3pt] {}; 
        \fill[black] (xpsi) circle (2pt) node[above right] {$x(\psi)$}; 

        \draw[gray] (50:3cm) -- node[pos=0.3, sloped, below] {$r_{\mathrm{out}}$} (0,0) -- +(0.5,0);
        \draw[gray, thick] (0:0.4) arc(0:50:0.4) node[pos=0.7, right] {\small $\psi$};

        \coordinate (xm) at ($(c) + 0.5*(1,4)$); 
        \coordinate (xstar) at ($(xm) + (120:1.5mm)$); 
        
        \draw[red] (xpsi) -- node[pos=0.45, above, sloped] {$L(\psi, \omega)$} ($(xstar)!-.5!(xpsi)$) coordinate (x'); 
        \fill (xstar) circle (1pt) node[above] {$y_{\tau}$};

        \draw[red, ->, >=stealth] ($(xpsi) + (-0.3,0)$) -- node[above, pos=0.5] {\small $\theta(\omega)$} +(-155:5mm);

    \end{scope}
    \end{tikzpicture}
    \caption{Measurement geometry with disjoint rods $B_q$ in the moderator $\Omega_0$. The angle $\psi$ specifies $x(\psi)$, and $\omega$ specifies the direction of $L(\psi,\omega)$.}
    \label{fig:Detector-Fuelrod-setup}
\end{figure}
\subsection{Tangencies of the rods}\label{ssec:tangencies}
We regard the angular parameter space as the torus
$(\mathbb R/2\pi\mathbb Z)^2$ and use $[0,2\pi)^2$ for
representatives of its points.

For a rod $B_p$, we decompose the vector $c_p-x(\psi)$ in the orthonormal frame
$\{\theta(\omega),\theta^\perp(\omega)\}$ by setting
\begin{equation}\label{def:t_p,n_p}
    t_p(\psi,\omega)
    :=
    \theta(\omega)\cdot\bigl(c_p-x(\psi)\bigr),
    \qquad
    n_p(\psi,\omega)
    :=
    \theta^\perp(\omega)\cdot\bigl(c_p-x(\psi)\bigr).
\end{equation}
The point $x(\psi)+t_p(\psi,\omega)\theta(\omega)$ is the orthogonal projection of $c_p$ onto $L(\psi,\omega)$. Thus, $|n_p(\psi,\omega)|$ is the distance from $c_p$ to $L(\psi,\omega)$. The quantity $|t_p(\psi,\omega)|$ is the distance from $x(\psi)$ to the projection point.
The line $L(\psi,\omega)$ intersects $B_p$ if $|n_p(\psi,\omega)|<r$, is tangent to $\partial B_p$ if $|n_p(\psi,\omega)|=r$, and misses $\overline{B_p}$ if $|n_p(\psi,\omega)|>r$.

We define the \emph{tangency set} of $B_p$ by
\begin{equation}\label{eq:T_p_def}
    \mcT_p
    :=
    \bigl\{
        (\psi,\omega)\in(\mathbb R/2\pi\mathbb Z)^2:
        |n_p(\psi,\omega)|=r,\;
        t_p(\psi,\omega)>0
    \bigr\}.
\end{equation}
For $\tau=(\psi,\omega)\in\mcT_p$, the tangency point on $\partial B_p$ is
\begin{equation}\label{eq:tangency_point}
    y_\tau
    :=
    x(\psi)+t_p(\psi,\omega)\theta(\omega)
    \in\partial B_p.
\end{equation}
If a line is tangent to more than one rod,
$y_\tau$ denotes the tangency point on the specified rod.

At $\tau=(\psi,\omega)\in\mcT_p$, we write $t_p=t_p(\psi,\omega)>0$. Since $y_\tau\in\Omega$ and
$|x(\psi)|=\ro$, we have
\[
    2t_p\,x(\psi)\cdot\theta(\omega)+t_p^2
    =|y_\tau|^2-\ro^2<0.
\]
Hence $x(\psi)\cdot\theta(\omega)<0$. Thus, $\theta(\omega)$ points into $\Omega$ at $x(\psi)$. The intersection $L(\psi,\omega)\cap\Omega$ therefore lies on the forward ray from $x(\psi)$ in direction $\theta(\omega)$. If $L(\psi,\omega)$ is also tangent to $\partial B_q$, the tangency point on $\partial B_q$ lies in $\Omega$. Consequently, $t_q(\psi,\omega)>0$.

Equivalently, the tangency set is given by
\[
    \mcT_p = \{ G_p = 0 \} \cap \{t_p > 0 \},
\]
where $G_p(\psi,\omega) := n_p(\psi,\omega)^2-r^2.$ 
The function $G_p$ is negative exactly when $L(\psi,\omega)$ meets the interior of $B_p$ and positive when it misses $\overline{B}_p$. Using
$\partial_\omega\theta = \theta^\perp$ and
$\partial_\omega\theta^\perp = -\theta$, we obtain by differentiation
\begin{equation} \label{eq:dGp}
    \partial_\omega G_p(\psi,\omega) = -2n_p(\psi,\omega) t_p(\psi,\omega).
\end{equation}
In particular, we have
\[
    |\partial_\omega G_p(\psi,\omega)| = 2r\,t_p(\psi,\omega)>0
\]
for any $(\psi,\omega)\in\mcT_p.$
\begin{definition}[Simple tangency]\label{def:simple_tangency}
Let $m\in\{1,\dots,K\}$ and $\tau=(\psi,\omega)\in\mcT_m$.
We call $\tau$ \emph{simple} if $L(\psi,\omega)$ is not tangent to $\partial B_q$ for any $q\ne m$. Equivalently, $\tau\in\mcT_m$ is simple if
\[
    \tau\notin\bigcup_{q\ne m}\mcT_q.
\]
We denote the set of simple tangencies of $B_m$ by $\mcS_m$.
\end{definition}

\begin{lemma}\label{lm:discrete_tangency}
Let $p,q\in\{1,\dots,K\}$. Then the tangency set $\mathcal{T}_p$ is a nonempty real-analytic one-dimensional submanifold of $(\mathbb R/2\pi\mathbb Z)^2$. If $p\neq q$, then $\mcT_p\cap\mcT_q$ consists of exactly eight points. Consequently, all but finitely many
tangencies of each rod are simple. In particular, we have
$\mcS_p\neq\emptyset$.
\end{lemma}
\begin{proof}
Equation~\eqref{eq:dGp} gives $\partial_\omega G_p\neq0$ on $\mcT_p$. Therefore, the real-analytic implicit function theorem (see, for instance, \cite[Theorem 2.3.5]{primer:book}) implies that $\mcT_p$ is locally a real-analytic curve. The set $\mcT_p$ is nonempty because every point of $\partial\Omega$ lies outside $\overline B_p$ and admits a tangent line to $\partial B_p$ directed toward the tangency point.

Now let $p\neq q$. A point of $\mcT_p\cap\mcT_q$ corresponds to a
line tangent to both $\partial B_p$ and $\partial B_q$. Since $|c_p-c_q|>2r$, the two circles have exactly four common tangent lines. Each common tangent line contains points of $\Omega$ and therefore crosses $\partial\Omega$ at exactly two points. For each orientation of a common tangent line, exactly one of these boundary points is the entry point into $\Omega$. The entry point determines $\psi$. The orientation determines $\omega$. Hence $|\mcT_p\cap\mcT_q|=4\cdot2=8$. Since there are finitely many rods, only finitely many points of $\mcT_p$ fail to be simple.
\end{proof}

To define one-sided angular limits, we need a neighborhood of a
tangency containing no other tangent direction when the boundary
point $x(\psi)$ is held fixed. The concept is illustrated in Figure~\ref{fig:delta-cone}.
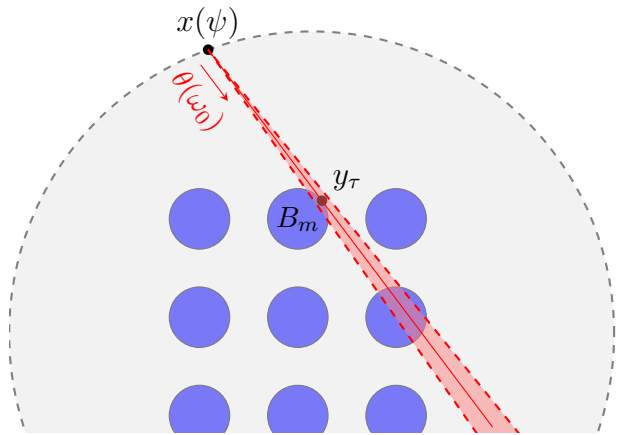
\begin{figure}[ht]
    \centering
    \begin{tikzpicture}
    \begin{scope}
        \path[clip] (-4,5) rectangle +(8.3,-6.3); 


        \def\p{110}
        \def\om{-53} 
        \def\omperp{\fpeval{\om+90}}
        \def\del{2} 
        \def\r{4mm} 
        
        \draw[thick, dashed, gray, fill=gray!20!white, fill opacity=0.5] (0,0) circle (4cm);
        \coordinate (xpsi) at (\p:4cm); 
        \fill[black] (xpsi) circle (2pt) node[above] {$x(\psi)$}; 

        \coordinate (xstar) at ($(xpsi) + (\om:2.5cm)$); 
        \coordinate (c) at ($(xstar) - (\omperp:\r)$); 
        \foreach \i in {-1, 0, 1}{
            \foreach \j in {-2, -1, 0}{
                \draw[black, fill=blue, opacity=0.5] ($(c) + 1.3*(\i,\j)$) circle (\r);
            }
        }
        \node at (c) {\small $B_m$}; 

        \fill (xstar) circle (2pt) node[above right] {$y_\tau$};

        \draw[dashed, thick, red, fill=red!50!white, fill opacity=0.5] ($(xpsi) + (\om-\del:10cm)$) -- (xpsi) -- ($(xpsi) + (\om+\del:10cm)$) -- cycle;
        
        \draw[red] (xpsi) -- ($(xstar)!-1.5!(xpsi)$) coordinate (x'); 

        \draw[red, ->, >=stealth] ($(xpsi) + (-0.1,-0.2)$) -- node[below, pos=0.6, sloped] {\small $\theta(\omega_0)$} +(\om:6mm);
        
    \end{scope}
    
    \end{tikzpicture}
    \caption{An angular neighborhood of a tangency
    $\tau=(\psi,\omega_0)$. With the boundary point $x(\psi)$ fixed,
    the direction $\omega_0$ is isolated among the directions tangent
    to any rod.}
    \label{fig:delta-cone}
\end{figure}
\begin{lemma}[Angular window]\label{lm:window}
Let $m\in\{1,\dots,K\}$ and let $\tau=(\psi,\omega_0)\in\mcT_m.$
Then there exists $\delta_\tau\in(0,\pi)$, depending only on the geometry
$\mcG$ and the tangency $\tau$, such that
\begin{equation}\label{eq:simple_tangency}
    (\psi,\omega)\notin\bigcup_{q=1}^K\mcT_q
    \qquad
     \text{whenever } \qquad 0<|\omega-\omega_0|<\delta_\tau.
\end{equation}
Moreover, for every $|\omega-\omega_0|<\delta_\tau$, we have
\[
    t_m(\psi,\omega)>0,
    \qquad
    \theta(\omega)\cdot x(\psi)<0.
\]
\end{lemma}
\begin{proof}
Fix $q\in\{1,\dots,K\}$. Since $\overline B_q\subset\Omega$ and $x(\psi)\in\partial\Omega$, we have $b_q = |c_q - x(\psi)| > r$. Choose $\alpha_q$ such that $c_q - x(\psi) = b_q \theta(\alpha_q)$. Then $n_q(\psi,\omega) = b_q\sin(\alpha_q-\omega)$, and we obtain
\[
    G_q(\psi,\omega) = b_q^2 \sin^2(\alpha_q-\omega) - r^2.
\]
Thus, the full line $L(\psi,\omega)$ is tangent to $\partial B_q$ if and only if 
\[
    |\sin(\alpha_q - \omega)| = \frac{r}{b_q}.
\]
Since $0 < r / b_q < 1$, this equation has only finitely many solutions
modulo $2\pi$. Membership in $\mcT_q$ also requires $t_q>0$. Since there are finitely many rods, the set of tangent directions from the fixed point $x(\psi)$ is finite, and $\omega_0$ is isolated in it.
Choose $\delta_\tau>0$ small enough to exclude all other such directions. By continuity and the inward-pointing property at $\tau$, it can also be chosen so that $\delta_\tau<\pi$ and
\[
    t_m(\psi,\omega)>0,
    \qquad \theta(\omega)\cdot x(\psi)<0,
    \ \text{when} \ |\omega-\omega_0|<\delta_\tau.\qedhere
\]
\end{proof}
The restriction $\delta_\tau<\pi$ allows the window to be represented
by a single interval without identifying two of its points modulo
$2\pi$. All angle differences and reflections below are taken
in this local coordinate.

We next identify the side of a tangency on which a small angular
perturbation causes the ray to enter the rod. This is the side from which the singular limits in Section~\ref{ssec:functionals} are taken.
\begin{definition}[Singular side]\label{def:side}
Let $m \in \{1,\dots,K\}$ and let $\tau = (\psi,\omega_0)\in\mcT_m.$ We define
\begin{equation}\label{eq:epsilon_tau}
    \etan := \sgn n_m(\psi,\omega_0) \in\{+1,-1\}.
\end{equation}
This is well defined because $|n_m(\psi,\omega_0)| = r>0$. The \emph{singular side} of $\tau$ is the set of directions $\omega$ in the angular window of Lemma~\ref{lm:window} with $\etan(\omega-\omega_0)>0$. By Lemma~\ref{lm:side}(ii), this is the side on which
$L(\psi,\omega)$ enters $B_m$. At a simple tangency with
$F_m^{(a,f)}(\tau)\ne0$, Proposition~\ref{prop:HR} identifies
the nonzero square-root singular coefficient of the data on
this side. The term singular side refers to this geometric
choice, including when the leading coefficient vanishes. We write $\omega\to\omega_0^{\etan}$ for the one-sided limit from the singular side.
\end{definition}

The following lemma identifies the inward normal at the tangency
point and verifies that the singular side is the side on which the
perturbed line enters the rod.
\begin{lemma}\label{lm:side}
Let $\tau=(\psi,\omega_0)\in\mcT_m$, let $y_\tau$ be its tangency point,
and let $\delta_\tau$ be as in Lemma~\ref{lm:window}. Then the following statements hold.
\begin{enumerate}[label=\textup{(\roman*)}]
    \item\label{it:side_geom} We have $ c_m-y_\tau
        =
        \etan r\,\theta^\perp(\omega_0).$
    Thus, $\etan \theta^\perp(\omega_0)$ is the unit normal to
    $\partial B_m$ at $y_\tau$ pointing towards the center $c_m$.

    \item\label{it:side_cut}
    For $ 0<|\omega-\omega_0|<\delta_\tau,$
    the line $L(x(\psi),\theta(\omega) )$ intersects the interior of
    $B_m$ if and only if $ \etan (\omega-\omega_0)>0.$
\end{enumerate}
\end{lemma}
\begin{proof}
\ref{it:side_geom} Since the line $L(x(\psi),\theta(\omega_0))$ is tangent to $\partial B_m$ at $y_\tau$, we have $\theta(\omega_0)\cdot(c_m-y_\tau)=0$.
The relation $y_\tau\in\partial B_m$ gives $|c_m-y_\tau|=r$.
Consequently, $c_m-y_\tau=\pm r\,\theta^\perp(\omega_0)$.  Using 
$ y_\tau= x(\psi)+t_m(\psi,\omega_0)\theta(\omega_0)$ and taking the inner product with
$\theta^\perp(\omega_0)$ gives
\[
\theta^\perp(\omega_0)\cdot(c_m-y_\tau)=\theta^\perp(\omega_0)\cdot(c_m-x(\psi))=n_m(\psi,\omega_0).
\] By the definition of $\epsilon_\tau$, we have $n_m(\psi,\omega_0)=\etan r$. This proves \ref{it:side_geom}.

\ref{it:side_cut} Recall that $G_m(\psi,\omega) =
    n_m(\psi,\omega)^2-r^2.$
The line $L(x(\psi),\theta(\omega))$ intersects the interior of $B_m$
if and only if $ G_m(\psi,\omega)<0.$
At the tangency, $ G_m(\psi,\omega_0)=0.$ Evaluating \eqref{eq:dGp} at this configuration gives
\begin{align*}
    \partial_\omega G_m(\psi,\omega_0) &= -2n_m(\psi,\omega_0) t_m(\psi,\omega_0) =
    -2 \etan r\,t_m(\psi,\omega_0) \neq 0,
\end{align*}
because $t_m(\psi,\omega_0)>0$. The nonzero derivative determines the sign of $G_m$ near $\omega_0$. On each component of the punctured window, $G_m$ has no zero. Indeed, since $t_m>0$ there, a zero of $G_m$ would imply $(\psi,\omega)\in\mcT_m$, contradicting Lemma~\ref{lm:window}. The sign of $G_m$ is therefore constant on each component and agrees throughout the window with the sign of $-\epsilon_\tau(\omega-\omega_0)$. We conclude that
\[
    G_m(\psi,\omega)<0
    \quad\Longleftrightarrow\quad
    \epsilon_\tau(\omega-\omega_0)>0,
\]
which proves \ref{it:side_cut}.
\end{proof}
We collect the geometric quantities attached to a tangency.
\begin{notation}\label{not:tangency_data}
Let $m\in\{1,\dots,K\}$. For every $\tau =(\psi,\omega) \in \mcT_m$, we write
\[
    x_\tau:=x(\psi),
    \qquad
    \theta_\tau:=\theta(\omega),
    \qquad
    L_\tau:=t_m(\psi,\omega).
\]
At a configuration tangent to more than one rod, the notation is understood relative to the specified rod. The tangency point $y_\tau$ is defined by \eqref{eq:tangency_point}. For each tangency under consideration, fix a choice of $\delta_\tau$ satisfying Lemma~\ref{lm:window}. We use $\epsilon_\tau$ as defined in \eqref{eq:epsilon_tau}. These definitions give
\[
    y_\tau = x_\tau + L_\tau \theta_\tau, \qquad L_\tau=|y_\tau-x_\tau|>0.
\]
\end{notation}
\begin{lemma}\label{lm:L_bounds}
Let $m\in\{1,\dots,K\}$. For every $\tau \in \mcT_m$, we have
\[0<L_{\min}\le L_{\tau}\le L_{\max}<\infty,
\]
where \[ L_{\min} := \min_{1\le q\le K}\bigl(\ro-|c_q|-r\bigr),\quad L_{\max} := 2\ro.\]
The bounds $L_{\min}$ and $L_{\max}$ depend only on the geometry $\mcG$.
\end{lemma}

\begin{proof}
Since $x_\tau\in\partial\Omega$ and $y_\tau\in\partial B_m$, we obtain the lower bound
\[
    L_\tau
    =
    |y_\tau-x_\tau|
    \ge
    \operatorname{dist}(\partial\Omega,\overline B_m)
    =
    \ro-|c_m|-r
    \ge L_{\min}>0.
\]
Since  $x_\tau,y_\tau\in\overline\Omega$, we have
$L_\tau\le\operatorname{diam}(\Omega)=2\ro=L_{\max}$, which gives the upper bound.
\end{proof}

\subsection{Chords along a backward ray}\label{ssec:chords}
The representation \eqref{eq:u_backward} expresses the solution $u$ of \eqref{eq:transport} at $y$ as an integral along the backward ray from $y$ in direction $-\theta(\omega)$.
We describe the entry parameters and chord lengths for the rods crossed by this ray.

For $y\in\R^2$ and $\omega\in[0,2\pi)$, we define the \emph{backward ray}
\begin{equation}\label{def:back-ray}
    \mcR(y,\omega):=\bigl\{y-s\,\theta(\omega):s\ge0\bigr\}.
\end{equation}
For $q\in\{1,\dots,K\}$, we set
\begin{equation}\label{def:rho_sigma}
    \sigma_q(y,\omega):=-\bigl(c_q-y\bigr)\cdot\theta(\omega),
    \qquad
    \rho_q(y,\omega):=(c_q-y)\cdot\theta^\perp(\omega).
\end{equation}
The quantity $|\rho_q(y,\omega)|$ is the distance from $c_q$ to the line $L(y,\theta(\omega))$. The quantity $\sigma_q(y,\omega)$ is the coordinate of $c_q-y$ in the direction $-\theta(\omega)$. The functions $\rho_q$ and $\sigma_q$ are polynomials in the coordinates of $y$, $\cos\omega$ and $\sin\omega$, and are therefore real-analytic. If $y\in L(\psi,\omega)$, then $\rho_q(y,\omega)=n_q(\psi,\omega)$. The orthogonal decomposition gives $|c_q-y|^2=\sigma_q(y,\omega)^2+\rho_q(y,\omega)^2$.
\begin{figure}[ht]
    \centering
    \begin{tikzpicture}
    \begin{scope}[scale=1.0]
        \path[clip] (-2,5) rectangle (10.5,-2); 

    
        \coordinate (cp) at (0,0);
        \coordinate (cq) at (6,3);
        \def\r{13mm} 
        \def\t{96} 
        \def\tperp{\fpeval{\t+90}}
        \coordinate (xstar) at ($(cp) +(\t:\r)$); 

        \draw[black, fill=blue, opacity=0.5, thick] (cp) circle (\r);
        \node[below left=.1*\r of cp] {\large $B_m$};
        
        \draw[black, fill=blue, opacity=0.5, thick, name path=Omegaq] (cq) circle (\r);
        \fill (cq) circle (2pt) node[above] {$c_q$};
        \node[above right=.1*\r of cq] {\large $B_q$};
        \draw[white, dashed] (cq) -- node[pos=0.5,above] {$r$} +(155:\r);

        
        \coordinate (y) at ($(xstar) - (\tperp:8cm)$);
        \coordinate (xpsi) at ($(xstar) - (\tperp:10cm)$);
        \draw[red, dashed, name path=ray] (y) -- node[pos=0.6, below, sloped] {$\mcR(y,\omega)$} ($(xstar)!-0!(y)$) coordinate (y'); 
        \fill (xstar) circle (2pt) node[above] {$y$};
        \fill[gray!50!black] (xpsi) circle (2pt) node[above] {$x(\psi)$};
        \draw[gray, dotted] (xpsi) -- +(\tperp:2cm);

        \draw[red, ->, >=stealth] (xstar) -- +(\tperp:6mm) node[left] {\small $\theta(\omega)$};

        \draw[blue] ($(xstar)!(cq)!(y)$) coordinate (rho) -- node[pos=0.5, right] {$|\rho_{q}|$} (cq);
        \node[red, rotate=\tperp] at (rho) {\small |};
        \node[below=4pt, red] at (rho) {$\sigma_{q}$};

        \path[name intersections={of=Omegaq and ray, by={sminus, splus}}];
        \draw[red, thick, |-|] (splus) node[below] {\small $s_+$} -- (sminus) node[below] {\small $s_-$} ;

        \path[name path=rho_pq] (cq) -- ($(rho)!-3!(cq)$);
        \path [name intersections={of=rho_pq and Omegaq, by=ystar}];
    \end{scope}
    
    \end{tikzpicture}
    \caption{Geometry of the intersection of the backward ray
    $\mcR(y,\omega)$ with a rod $B_q$.  When the backward ray crosses $B_q$ transversely, the parameters $s_-$ and $s_+$ give the distances from $y$ to the first and second intersection points with $\partial B_q$, respectively.}
    \label{fig:sigma_rho}
\end{figure}
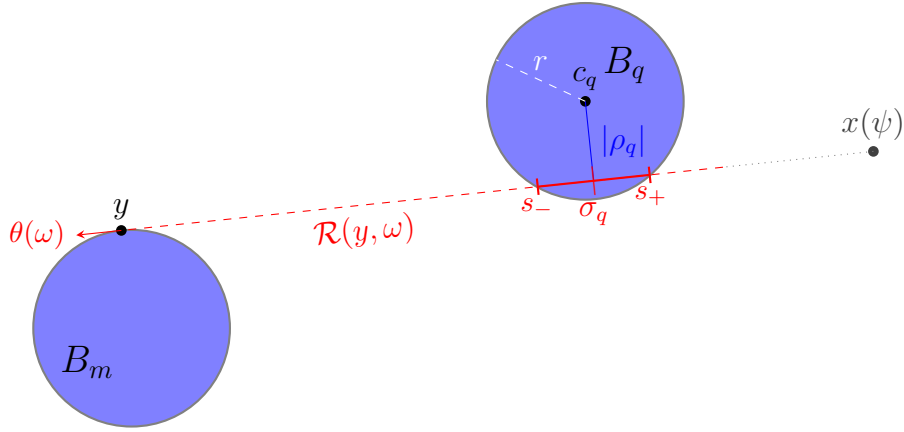
\begin{lemma}\label{lm:d_q_ell_q}
Let $\omega\in[0,2\pi)$, $y\in\R^2$ and $q\in\{1,\dots,K\}$ satisfy $|c_q-y|>r$. Then the following statements hold.
\begin{enumerate}[label=\textup{(\roman*)}]
\item\label{itm:chord_condi} The line $L(y,\theta(\omega))$ meets $\partial B_q$ if and only if $|\rho_q(y,\omega)|\le r$.
\item\label{itm:chord_tangent}
The line $L(y,\theta(\omega))$ is tangent to $\partial B_q$ if and only if $|\rho_q(y,\omega)|=r$.
\item\label{itm:chord_trans} The backward ray $\mcR(y,\omega)$ crosses $B_q$ transversely if and
only if
\begin{equation}\label{eq:backward_crossing_condition}
    |\rho_q(y,\omega)|<r
    \qquad\text{and}\qquad
    \sigma_q(y,\omega)>0.
\end{equation}
Under \eqref{eq:backward_crossing_condition}, let $d_q(y,\omega)$ denote the smallest $s\ge0$ for which $y-s\theta(\omega)\in\partial B_q$. Let $\ell_q(y,\omega)$ denote the length of the chord $\mcR(y,\omega)\cap B_q$.
These quantities are given by
\begin{equation}\label{eq:def:d_q:ell_q}
    d_q(y,\omega)
    =
    \sigma_q(y,\omega)
    -
    \sqrt{r^2-\rho_q(y,\omega)^2},
    \qquad
    \ell_q(y,\omega)
    =
    2\sqrt{r^2-\rho_q(y,\omega)^2}.
\end{equation}
\end{enumerate}
\end{lemma}

\begin{proof}
Write $\rho_q=\rho_q(y,\omega)$ and
$\sigma_q=\sigma_q(y,\omega)$.
For $s\in\R$, the orthogonal decomposition gives
\[
    |y-s\theta(\omega)-c_q|^2
    =(s-\sigma_q)^2+\rho_q^2.
\]

\ref{itm:chord_condi}
The line $L(y,\theta(\omega))$ meets $\partial B_q$
if and only if $(s-\sigma_q)^2=r^2-\rho_q^2$
has a real solution.
This is equivalent to $|\rho_q|\le r$.

\ref{itm:chord_tangent}
When $|\rho_q|\le r$, the intersection parameters are
\[
    s_\pm=\sigma_q\pm\sqrt{r^2-\rho_q^2}.
\]
The parameters $s_-$ and $s_+$ coincide exactly when
$|\rho_q|=r$.
Hence, $L(y,\theta(\omega))$ is tangent to $\partial B_q$
if and only if $|\rho_q|=r$.

\ref{itm:chord_trans}
Suppose that $|\rho_q|<r$.
Since $|c_q-y|>r$, we obtain
\[
    \sigma_q^2
    =|c_q-y|^2-\rho_q^2
    >r^2-\rho_q^2>0.
\]
Thus, $s_-$ and $s_+$ both have the sign of $\sigma_q$.
The points $y-s\theta(\omega)$ inside $B_q$ correspond
to $s_-<s<s_+$.
Since the backward ray $\mcR(y,\omega)$ is parametrized
by $s\ge0$, the transverse crossing condition is
precisely \eqref{eq:backward_crossing_condition}.
Under this condition, the entry parameter is $d_q=s_-$.
The exit parameter is $s_+$.
Since $|\theta(\omega)|=1$, the chord length is
$\ell_q=s_+-s_-$.
Substituting the expressions for $s_\pm$ gives
\eqref{eq:def:d_q:ell_q}.
\end{proof}
The chord geometry is illustrated in Figure~\ref{fig:sigma_rho}.
\begin{remark}\label{rm:chord_analytic}
   Fix $q\in\{1,\dots,K\}$ and use the $2\pi$-periodic
extensions of $\rho_q$ and $\sigma_q$ in the angular
variable. The formulas \eqref{eq:def:d_q:ell_q} define
real-analytic functions $d_q$ and $\ell_q$ on the
open set\begin{equation}\label{eq:transverse_set}
    \bigl\{
        (y,\omega)\in\R^2\times\R
        \bigm|
        |c_q-y|>r,\ 
        |\rho_q(y,\omega)|<r
    \bigr\}.
\end{equation}
Indeed, $\rho_q$ and $\sigma_q$ are real-analytic. For every $(y,\omega)$ in \eqref{eq:transverse_set},
the orthogonal decomposition of $c_q-y$ gives
\[
    \sigma_q(y,\omega)^2
    =
    |c_q-y|^2-\rho_q(y,\omega)^2
    >
    r^2-\rho_q(y,\omega)^2
    >
    0.
\]
Hence, $\sigma_q(y,\omega)\neq0$ throughout
\eqref{eq:transverse_set}.
\end{remark}
For $y\in\R^2$ and $\omega\in[0,2\pi)$, we collect the indices of rods whose closed disks do not contain $y$ and whose interiors are crossed by $\mcR(y,\omega)$ in the set
\begin{equation}\label{eq:def:C_omega}
    \mcC(y,\omega)
    :=
    \bigl\{
        q\in\{1,\dots,K\}:
        |c_q-y|>r,\ 
        |\rho_q(y,\omega)|<r,\ 
        \sigma_q(y,\omega)>0
    \bigr\}.
\end{equation}
If $\mcC(y,\omega)$ is nonempty, write $\mcC(y,\omega)=\{q_1,\dots,q_N\}$
and order the indices according to the entry parameters along $\mcR(y,\omega)$, so that
\begin{equation}\label{eq:ordering}
    0<d_{q_1}(y,\omega)
    <
    \cdots
    <
    d_{q_N}(y,\omega).
\end{equation}
The inequalities are strict because the rods are pairwise disjoint.

If $y$ is a tangency point of $B_m$, then for every $q\neq m$, the reverse triangle inequality gives
\[
    |c_q-y|
    \ge
    |c_q-c_m|-r
    >
    r.
\]
Hence, Lemma~\ref{lm:d_q_ell_q} applies to every $q\neq m$.
The rod $B_m$ itself is treated separately when needed.

\subsection{Parameter space}\label{sec:parameter}


\begin{definition}\label{def:parameter space}
Fix a known moderator attenuation $a_0>0$. Let $\mcG$ be as in \eqref{eq:geometry}, and let $0<a_{\min}\le a_{\max}$ and
$0\le f_{\min}\le f_{\max}$. The \emph{parameter space}
$X=X(a_{\min},a_{\max},f_{\min},f_{\max},a_0, \mcG)$ consists of pairs $(a,f)\in L^\infty(\Omega)\times L^\infty(\Omega)$ of the form
\begin{equation}\label{eq:a_f_chi}
    a=a_0\chi_{\Omega_0}+\sum_{j=1}^{K}a_j\chi_{B_j},
    \qquad
    f=\sum_{j=1}^{K}f_j\chi_{B_j},
\end{equation}
with
\[
    a_{\min}\le a_j\le a_{\max},
    \qquad
    f_{\min}\le f_j\le f_{\max},
    \qquad j=1,\dots,K.
\]
In the whole-space formulas \eqref{eq:beam}--\eqref{eq:u_backward}, we extend both $a$ and $f$ by zero outside $\Omega$.
\end{definition}
Thus, the moderator has known attenuation $a_0$ and does not emit. Each rod carries one unknown attenuation value and one unknown source value.
\begin{remark}\label{rm:X_compact}
For fixed $\mcG$, the representation \eqref{eq:a_f_chi} is unique
up to sets of measure zero. The coefficient map $(a,f) \longmapsto  (a_1,\dots,a_K,f_1,\dots,f_K)$ therefore identifies $X$ with the compact box $ [a_{\min},a_{\max}]^K \times [f_{\min},f_{\max}]^K \subset\mathbb R^{2K}.$ In particular, $X$ is compact. We use $a,f$ for both the functions and their coefficient vectors. Since two elements of $X$ share the moderator value $a_0$, the coefficient and function norms satisfy
\[
    \|a-\tilde a\|_{\ell^\infty}
    =\|a-\tilde a\|_{L^\infty(\Omega)},
    \qquad
    \|f-\tilde f\|_{\ell^\infty}
    =\|f-\tilde f\|_{L^\infty(\Omega)}.
\]
 We equip $X$ with the metric $d_X\bigl((a,f),(\tilde a,\tilde f)\bigr):=\|a-\tilde a\|_{\ell^\infty}+\|f-\tilde f\|_{\ell^\infty}$ induced by $L^\infty(\Omega)\times L^\infty(\Omega)$. We note that $X$ is not a vector space.
\end{remark}

\subsection{The solution of the transport equation and its tangential trace}\label{ssec:utan}

Since $a$ and $f$ are piecewise constant, the integral in
\eqref{eq:u_backward} reduces to a finite sum over the chords
described in Subsection~\ref{ssec:chords}. Fix a rod $B_m$ and a tangency $\tau=(\psi,\omega_0)\in\mcT_m$, with $y_\tau$ and $\theta_\tau$ as in Notation~\ref{not:tangency_data}. Throughout this subsection, $y_\tau$ remains fixed while $\omega$ varies. For $(a,f)\in X$ and $q\in\mcC(y_\tau,\omega)$, we define
\begin{equation}\label{eq:def:beta_W}
    \beta_{q,m}(\omega):=\int_0^{d_q(y_\tau,\omega)}
            a\bigl(y_\tau-\xi\theta(\omega)\bigr)\,d\xi,
    \qquad
    W_{q,m}(\omega):=f_q\, \frac{1-e^{-a_q\ell_q(y_\tau,\omega)}}{a_q}.
\end{equation}
The quantity $\beta_{q,m}(\omega)$ is the attenuation accumulated
along the segment from $y_\tau$ to $y_\tau-d_q(y_\tau,\omega)\theta(\omega)$. The quantity $W_{q,m}(\omega)$ is the source contribution from the chord in $B_q$, including attenuation within $B_q$. We suppress the dependence of $\beta_{q,m}$ and $W_{q,m}$
on $(a,f)$ and the fixed tangency $\tau$. The integral defining $\beta_{q,m}$ is finite because $a$ is bounded and $d_q(y_\tau,\omega)$ is finite. The denominator in $W_{q,m}$ is positive because $a_q\ge a_{\min}>0$.

Since $y_\tau\in\partial B_m$, the backward ray $\mcR(y_\tau,\omega)$ crosses $B_m$ precisely when $\sigma_m(y_\tau,\omega)>0$. Under this condition, the chord $\mcR(y_\tau,\omega)\cap B_m$ has length $2\sigma_m(y_\tau,\omega)$. We set
\[
    \ell_m^+(\omega)
    :=
    2\max\{\sigma_m(y_\tau,\omega),0\}.
\]
\begin{lemma}\label{lm:u_tan_as_sum}
For every $(a,f)\in X$ and every $\omega\in[0,2\pi)$, the solution $u$ of \eqref{eq:transport}, given by \eqref{eq:u_backward}, satisfies
\begin{equation}\label{eq:u_tan_sum}
    u\bigl(y_\tau,\theta(\omega)\bigr)
    =
    f_m\,\frac{1-e^{-a_m\ell_m^+(\omega)}}{a_m}
    +
    \sum_{q\in\mcC(y_\tau,\omega)}
        W_{q,m}(\omega)e^{-\beta_{q,m}(\omega)}.
\end{equation}
\end{lemma}
\begin{proof}
    By \eqref{eq:u_backward}, we have
    \[
    u(y_\tau,\theta(\omega))
    =
    \int_0^\infty
        f\bigl(y_\tau-s\theta(\omega)\bigr)
        \exp\left(
            -\int_0^s
                a\bigl(y_\tau-\xi\theta(\omega)\bigr)\,d\xi
        \right)
    \,ds.
\]
The source vanishes outside the rods. Hence only the portions of $\mcR(y_\tau,\omega)$ inside the rods contribute to the integral. The contribution from $B_m$ is
\[
    \int_0^{\ell_m^+(\omega)}f_m e^{-a_ms}\,ds
    =
    f_m\,\frac{1-e^{-a_m\ell_m^+(\omega)}}{a_m}.
\]
This expression is zero when the backward ray does not
enter $B_m$. 

For $q\in\mcC(y_\tau,\omega)$, the contribution from $B_q$
is the integral over
\[
    [d_q(y_\tau,\omega),
      d_q(y_\tau,\omega)+\ell_q(y_\tau,\omega)].
\]
Set $s=d_q(y_\tau,\omega)+\kappa$. Since $a=a_q$ inside
$B_q$, this contribution equals
\begin{align*}
    &\int_{d_q(y_\tau,\omega)}^{
        d_q(y_\tau,\omega)+\ell_q(y_\tau,\omega)}
        f_q
        \exp\left(
            -\int_0^s
                a\bigl(y_\tau-\xi\theta(\omega)\bigr)\,d\xi
        \right)
    \,ds\\
    &\qquad=
    f_q e^{-\beta_{q,m}(\omega)}
    \int_0^{\ell_q(y_\tau,\omega)}
        e^{-a_q\kappa}\,d\kappa=
    W_{q,m}(\omega)e^{-\beta_{q,m}(\omega)}.
\end{align*}
The rod intervals are pairwise disjoint. Summing their
contributions proves \eqref{eq:u_tan_sum}.
\end{proof}
\begin{definition}[Tangential trace]\label{def:u_tan}
Let $(a,f)\in X$ and $\tau=(\psi,\omega_0)\in\mcT_m$. Let $u$ be the solution of \eqref{eq:transport} given by \eqref{eq:u_backward}. We define the \emph{tangential trace} of $u$ at $\tau$ by
\begin{equation}\label{eq:u_tan_limit}
    \utan[a,f](y_\tau,\theta_\tau):=u(y_\tau,\theta_\tau).
\end{equation}
\end{definition}
\begin{corollary}\label{cor:u_tan_value}
Let $\tau=(\psi,\omega_0)\in\mcS_m$ and $(a,f)\in X$.
Then there is an open interval $U\ni\omega_0$ with the following properties.
\begin{enumerate}[label=\textup{(\roman*)}]
\item\label{it:cor_const} For every $\omega\in U$, we have $\mcC(y_\tau,\omega) =\mcC(y_\tau,\omega_0)$. The ordering \eqref{eq:ordering} of the indices in $\mcC(y_\tau,\omega_0)$ is unchanged on $U$.
\item\label{it:cor_cont} For every $q\in\mcC(y_\tau,\omega_0)$, the functions $d_q(y_\tau,\cdot)$, $\ell_q(y_\tau,\cdot)$, $\beta_{q,m}$ and $W_{q,m}$ are continuous on $U$. The function $\ell_m^+$ is continuous on $U$. At the tangency, $\ell_m^+(\omega_0)=0$.
\item\label{it:cor_value}   The function $\omega\mapsto u(y_\tau,\theta(\omega))$ is continuous on $U$. The tangential trace is given by
    \begin{equation}\label{eq:u_tan_value}
        \utan[a,f](y_\tau,\theta_\tau)
        =
        \sum_{q\in\mcC(y_\tau,\omega_0)}
            W_{q,m}(\omega_0)e^{-\beta_{q,m}(\omega_0)}.
    \end{equation}
\end{enumerate}
\end{corollary}

\begin{proof}
\ref{it:cor_const} Fix $q\ne m$. Since $\tau$ is simple, we have
$|\rho_q(y_\tau,\omega_0)|=|n_q(\psi,\omega_0)|\ne r$. By continuity, the sign of $\rho_q(y_\tau,\omega)^2-r^2$ is constant near $\omega_0$.
If $|\rho_q(y_\tau,\omega_0)|<r$, then $\sigma_q(y_\tau,\omega_0)\ne0$ by Remark~\ref{rm:chord_analytic}. The sign of $\sigma_q(y_\tau,\omega)$ is then also constant on a sufficiently small neighborhood of $\omega_0$.
The crossing criterion \eqref{eq:backward_crossing_condition}
therefore shows that membership of $q$ in $\mcC(y_\tau,\omega)$ is unchanged on this neighborhood. The index $m$ is excluded because $|c_m-y_\tau|=r$.

Since there are finitely many rods, we can choose a common open interval $U\ni\omega_0$ on which $\mcC(y_\tau,\omega)=\mcC(y_\tau,\omega_0)$.
The entry parameters $d_q(y_\tau,\cdot)$, $q\in\mcC(y_\tau,\omega_0)$, are continuous on $U$ by Remark~\ref{rm:chord_analytic}. After shrinking $U$ if necessary, the strict inequalities in \eqref{eq:ordering} remain valid.

\ref{it:cor_cont} Remark~\ref{rm:chord_analytic} also gives continuity of
$\ell_q(y_\tau,\cdot)$ for $q\in\mcC(y_\tau,\omega_0)$. The formula defining $W_{q,m}$ consequently shows that $W_{q,m}$ is continuous on $U$.
The identity $\ell_m^+(\omega)=2\max\{\sigma_m(y_\tau,\omega),0\}$
gives the same conclusion for $\ell_m^+$. Since $\sigma_m(y_\tau,\omega_0)=0$ at the tangency, we obtain $\ell_m^+(\omega_0)=0$.

If $\mcC(y_\tau,\omega_0)$ is nonempty, let $q_1,\dots,q_N$
be ordered as in \eqref{eq:ordering}. For $\omega\in U$ and $1\le k\le N$, the segment from
$y_\tau$ to $y_\tau-d_{q_k}(y_\tau,\omega)\theta(\omega)$ lies in $\Omega$ by convexity. The intersection of this segment with $B_m$ has length $\ell_m^+(\omega)$. For each $j<k$, the segment contains the full chord through $B_{q_j}$, of length $\ell_{q_j}(y_\tau,\omega)$. The segment meets no other rod interior. 

Splitting the attenuation integral into moderator and rod contributions, we obtain
\[
    \beta_{q_k,m}(\omega)
    =
    a_0d_{q_k}(y_\tau,\omega)
    +(a_m-a_0)\ell_m^+(\omega)
    +\sum_{j<k}(a_{q_j}-a_0)\ell_{q_j}(y_\tau,\omega), \qquad \omega\in U.
\]
The continuity of the terms on the right-hand side gives continuity of $\beta_{q_k,m}$ on $U$.

\ref{it:cor_value} By parts~\ref{it:cor_const} and \ref{it:cor_cont}, every term in \eqref{eq:u_tan_sum} is continuous on $U$.
Hence $\omega\mapsto u(y_\tau,\theta(\omega))$ is continuous on $U$.
Evaluating \eqref{eq:u_tan_sum} at $\omega=\omega_0$, where $\ell_m^+(\omega_0)=0$, gives \eqref{eq:u_tan_value}.
\end{proof}
\begin{remark}\label{rm:apriori_bound_utan}
For every $(a,f)\in X$, $y\in\R^2$ and $\theta\in\mathbb S^1$, the representation \eqref{eq:u_backward} gives
\begin{equation}\label{eq:u_tan_uniform_bound}
    0\le u(y,\theta)\le\frac{f_{\max}}{a_{\min}}.
\end{equation}
To see this, put $A(s)=\int_0^s a(y-\sigma\theta)\,d\sigma$. Then $A(0)=0$ and $A'(s)=a(y-s\theta)$ almost everywhere.
For every $T>0$, integration gives
\[
    \int_0^T A'(s)e^{-A(s)}\,ds
    =1-e^{-A(T)}
    \le1.
\]
Since the source is supported in the rods,
$f(y-s\theta)\le(f_{\max}/a_{\min})a(y-s\theta)$ almost everywhere. Consequently, we obtain
\[
    u(y,\theta)
    \le\frac{f_{\max}}{a_{\min}}
       \int_0^\infty A'(s)e^{-A(s)}\,ds
    \le\frac{f_{\max}}{a_{\min}}.
\]
The same bound holds for the tangential trace defined in Definition~\ref{def:u_tan}.
\end{remark}

\subsection{Data space}\label{sec:data}

\begin{definition}\label{def:data}
The data associated with $(a,f)\in X$ are
\[
    R_af(\psi,\omega):=R_af(x(\psi),\theta(\omega)),
    \qquad(\psi,\omega)\in\mcM,
\]
where $\mcM\subset[0,2\pi)^2$ is a measurement parameter set.
\end{definition}
For every selected simple tangency $\tau=(\psi,\omega_0)$ used in the stability estimate, we require
\begin{equation}\label{eq:measurement_windows}
    \{\psi\}\times(\omega_0-\delta_\tau,\omega_0+\delta_\tau)
    \subset\mcM.
\end{equation}
Angular intervals in \eqref{eq:measurement_windows} are understood modulo $2\pi$. The estimate therefore uses only finitely many boundary positions, with data on an angular interval at each selected position. For details, see \ref{ass:stab:data} in Assumption~\ref{ass:stability}.

\subsubsection{The singular functionals}\label{ssec:functionals}
Fix a simple tangency $\tau = (\psi,\omega_0) \in \mcS_m$, and define the punctured interval
\[
    I_\tau^\circ:=(\omega_0-\delta_\tau,\omega_0+\delta_\tau)
                   \setminus\{\omega_0\}.
\]
The boundary point $x_\tau$ is fixed while the direction $\omega$ varies. The notation $\omega\to\omega_0^{\epsilon_\tau}$ denotes approach from the intersecting side, where $\epsilon_\tau(\omega-\omega_0)>0$. The reflected angle $2\omega_0-\omega$ lies on the side where
$L(\psi,2\omega_0-\omega)$ misses $B_m$. The line $L(\psi,2\omega_0-\omega)$ may still intersect other rods.
Thus, the data at the reflected angle need not vanish.

\begin{definition}\label{def:functionals}
Let $\mcV_\tau$ be the space of real-valued functions $H\in C^2(I_\tau^\circ)$ for which the limits
\begin{align}
    \Phi_D[H]
    &:=
    \epsilon_\tau\lim_{\omega\to\omega_0^{\epsilon_\tau}}
       |\omega-\omega_0|^{1/2}H'(\omega),
       \label{eq:PhiD-def}\\
    \Phi_N[H]
    &:=
    -\lim_{\omega\to\omega_0^{\epsilon_\tau}}
    \left[
       |\omega-\omega_0|H''(\omega)
       +\frac{\epsilon_\tau}{2}
        \bigl(H'(\omega)-H'(2\omega_0-\omega)\bigr)
    \right]
    \label{eq:PhiN-def}
\end{align}
exist and are finite.
We define
\begin{equation}\label{eq:seminorm}
    \|H\|_\tau:=|\Phi_D[H]|+|\Phi_N[H]|
\end{equation}
for any $H \in \mcV_\tau$.
\end{definition}
Restricting $H$ to a smaller angular window leaves $\Phi_D[H]$, $\Phi_N[H]$ and $\|H\|_\tau$ unchanged. The expression $H'(2\omega_0-\omega)$ denotes $H'$ evaluated at the reflected angle $2\omega_0-\omega$. We suppress the dependence of $\Phi_D$ and $\Phi_N$ on the fixed tangency $\tau=(\psi,\omega_0)$. For the data functions in Proposition~\ref{prop:HR} with $\Phi_D[H]\ne0$, the sign conventions in \eqref{eq:PhiD-def} and \eqref{eq:PhiN-def} make the quotient in \eqref{eq:Delta_a:ND} equal to $a_m-a_0$. Section~\ref{sec:numerics} examines the numerical approximation of $\Phi_D[H]$ and $\Phi_N[H]$ in simple geometries.
\begin{remark}\label{rm:seminorm}
The space $\mcV_\tau$ is linear. The functionals $\Phi_D$ and $\Phi_N$ are linear on $\mcV_\tau$. Hence $\|\cdot\|_\tau$ is a seminorm. By \eqref{eq:seminorm}, we have
\[
    |\Phi_D[H]|\le\|H\|_\tau,
    \qquad
    |\Phi_N[H]|\le\|H\|_\tau.
\]
If $H$ extends to a $C^2$ function on a neighborhood of $\omega_0$, then $\Phi_D[H]=\Phi_N[H]=0$. Indeed, $H'$ and $H''$ are bounded near $\omega_0$, so $|\omega-\omega_0|^{1/2}H'(\omega)$ and $|\omega-\omega_0|H''(\omega)$ tend to zero. Continuity of $H'$ also gives $H'(\omega)-H'(2\omega_0-\omega)\to0$. Thus, $\|\cdot\|_\tau$ has a nontrivial kernel and is not a norm.
\end{remark}
For $(a,f)\in X$ and  $\tau=(\psi,\omega_0)\in\mcS_m$, we define
\begin{equation}\label{eq:F_m_def}
    F_m^{(a,f)}(\tau) :=
    f_m-(a_m-a_0)\,
    \utan[a,f](y_\tau,\theta_\tau).
\end{equation}
We call $F_m^{(a,f)}(\tau)$ the \emph{non-cancellation coefficient} of $(a,f)$ at $\tau$.

The following proposition specializes the local singularity calculations of Holman and Richardson \cite{Holman:Richardson:2023} to the circular interfaces of our model. Fix a simple tangency $\tau=(\psi,\omega_0)\in\mcS_m$. The circle $\partial B_m$ has curvature $1/r$. As $\omega$ varies, the line $L(x_\tau,\theta(\omega))$ rotates about $x_\tau$. The rotation center satisfies $|x_\tau-y_\tau|=L_\tau$. With the moderator value subtracted from the rod value, the source jump at $y_\tau$ is $f_m$. The attenuation jump at $y_\tau$ is $a_m-a_0$. The value $u(y_\tau,\theta_\tau)$ is the tangential trace defined in Definition~\ref{def:u_tan}.

The change of angular coordinate $\eta=\epsilon_\tau(\omega-\omega_0)$ places the intersecting side at $\eta>0$. Applying the local identities in this coordinate gives the signed
functionals in Definition~\ref{def:functionals}. The resulting formulas are stated below.
\begin{proposition}\label{prop:HR}
Let $(a,f)\in X$ and $\tau=(\psi,\omega_0)\in\mcS_m$, and set
$H(\omega):=R_af(x_\tau,\theta(\omega))$  for $0<|\omega-\omega_0|<\delta_\tau$. Then $H\in \mcV_\tau$. The functionals $\Phi_D$ and $\Phi_N$ satisfy
\begin{align}
    \Phi_D[H]&=\sqrt{2rL_\tau}\,F_m^{(a,f)}(\tau)\,e^{-Da(y_\tau,\theta_\tau)}, \label{eq:PhiD_H}\\
    \Phi_N[H]&=(a_m-a_0)\sqrt{2rL_\tau}\,\Phi_D[H]. \label{eq:PhiN_H}
\end{align}
On the nonintersecting side, we have
\[
    \lim_{\omega\to\omega_0^{-\epsilon_\tau}}
    |\omega-\omega_0|^{1/2}H'(\omega)=0.
\]
In particular, if $\Phi_D[H]\neq0$, then we obtain
\begin{equation}\label{eq:Delta_a:ND}
    a_m-a_0=\frac{\Phi_N[H]}{\sqrt{2rL_\tau}\,\Phi_D[H]}.
\end{equation}
\end{proposition}

\section{Uniqueness and reconstruction procedure}\label{sec:uniqueness}
By Proposition~\ref{prop:HR}, attenuation recovery reduces to finding, for each rod $B_p$, a simple tangency $\tau\in\mcS_p$ at which $F_p^{(a,f)}(\tau)$ is nonzero. For the corresponding data function $H(\omega)=R_af(x_\tau,\theta(\omega))$, we then have
$\Phi_D[H]\neq0$, and \eqref{eq:Delta_a:ND} gives $a_p$.
If $F_p^{(a,f)}(\tau)=0$, we call $\tau$ a \emph{cancellation point}. In this case, $\Phi_D[H]=\Phi_N[H]=0$ by
\eqref{eq:PhiD_H}--\eqref{eq:PhiN_H}, and the quotient
\eqref{eq:Delta_a:ND} is undefined.

We give a sufficient condition on the geometry and coefficients
for non-cancellation on a nonempty relatively open family of simple
tangencies of a rod. If this condition holds for every rod, the
quotient \eqref{eq:Delta_a:ND} determines all attenuation
coefficients from the full data.

\subsection{Non-cancellation along a rod boundary}
\label{ssec:noncancellation}

\begin{definition}[Non-cancellation condition]
\label{def:noncancellation}
Let $(a,f)\in X$ and $\tau=(\psi,\omega_0)\in\mcS_p$. We say that $(a,f)$ satisfies the \emph{non-cancellation condition} at $\tau$ if
\begin{equation}\label{eq:noncancellation:condtion}
    F_p^{(a,f)}(\tau)
    =
    f_p-(a_p-a_0)\,
    \utan[a,f](y_\tau,\theta_\tau)
    \neq0.
\end{equation}
\end{definition}
In \cite[(11)]{Holman:Richardson:2023}, the non-cancellation condition is stated in terms of $u(x^*,\theta^*)$ at the tangency point in the tangent direction. Equation \eqref{eq:noncancellation:condtion} expresses that condition in our jump convention.

Fix a rod $B_p$. For $\phi\in \mathbb R/2\pi\mathbb Z$, consider the directed line tangent to $\partial B_p$ at $c_p+r\theta(\phi)$ in direction $\theta^\perp(\phi)$. Let $x(\psi(\phi))$ be the point where the backward ray $\mcR(c_p+r\theta(\phi),\phi+\pi/2)$ leaves $\Omega$. This defines the configuration
\begin{equation}\label{eq:tangent:para}
 \tau(\phi):=\bigl(\psi(\phi),\phi+\tfrac{\pi}{2}\bigr)\in\mcT_p,\qquad
y_{\tau(\phi)}=c_p+r\theta(\phi),\quad \theta_{\tau(\phi)}=\theta^\perp(\phi),
\end{equation}
and all quantities attached to $\tau(\phi)$ are functions of $\phi$. Since $c_p-y_{\tau(\phi)}=-r\theta(\phi)=r\theta^\perp(\phi+\pi/2)$, Lemma~\ref{lm:side}\ref{it:side_geom} gives $\epsilon_{\tau(\phi)}=+1$. If $L(\tau(\phi))$ is tangent to more than one rod, $y_{\tau(\phi)}$ denotes the tangency point on $\partial B_p$.

The map $\phi\mapsto\tau(\phi)$ is a continuous bijection from $\mathbb R/2\pi\mathbb Z$ onto the branch
$\{\tau\in\mcT_p:\epsilon_\tau=+1\}$. The inverse is the continuous map $(\psi,\omega_0)\mapsto\omega_0-\pi/2$. Indeed, if $\epsilon_\tau=+1$, then $y_\tau=c_p+r\theta(\phi)$ for a unique $\phi$ modulo $2\pi$.
Lemma~\ref{lm:side}\ref{it:side_geom} then gives $\omega_0=\phi+\pi/2$ modulo $2\pi$.
The point $x_\tau$ is the exit point of the backward ray from $y_\tau$. The branch is relatively open in $\mcT_p$ because $\epsilon_\tau$ is locally constant. Hence, the image of a nonempty open interval under $\phi\mapsto\tau(\phi)$ is a nonempty relatively open subset of $\mcT_p$.

The next assumption uses a line tangent to both $B_p$ and a second
emitting rod. On one side of the common tangent, nearby tangent
lines to $B_p$ cut the second rod in a short chord. We use the
resulting variation of the tangential trace to establish
non-cancellation at nearby simple tangencies.

\begin{assumption}\label{ass:essential_noncancel}
For the rod $B_p$, assume the following.
\begin{enumerate}[label=\textup{(\roman*)}]
   \item\label{ass:ess_non:A1}
    The attenuation jump is nonzero, so that
    \[
        a_p-a_0\neq0.
    \]
    \item\label{ass:ess_non:A2}
  There exist an index $m\in\{1,\dots,K\}\setminus\{p\}$
and an angle $\phi_*\in\mathbb R/2\pi\mathbb Z$
with the following properties.
For the fixed configuration
$\tau_*:=\tau(\phi_*)=(\psi_*,\omega_*)$
defined by \eqref{eq:tangent:para}, we have
    \[
        |\rho_m(y_{\tau_*},\omega_*)|=r,\qquad
            \sigma_m(y_{\tau_*},\omega_*)>0,\qquad f_m\ne0.
\]
For every $q\notin\{p,m\}$, we have
    \[
        |\rho_q(y_{\tau_*},\omega_*)|\ne r.
    \]
    Thus, $L(\psi_*,\omega_*)$ is tangent to exactly $\partial B_p$ and $\partial B_m$.
    The tangency point on $\partial B_m$ lies on $\mcR(y_{\tau_*},\omega_*)$ at a positive distance from $y_{\tau_*}$.
    \item\label{ass:ess_non:A3}
    Every rod crossed by $\mcR(y_{\tau_*},\omega_*)$ is encountered before the tangency point on $\partial B_m$. More precisely, for every $q\in\mcC(y_{\tau_*},\omega_*)$, we have
    \begin{equation}\label{eq:last_obstacle}
        d_q(y_{\tau_*},\omega_*)
        +
        \ell_q(y_{\tau_*},\omega_*)
        <
        \sigma_m(y_{\tau_*},\omega_*).
    \end{equation}
\end{enumerate}
\end{assumption}
\begin{lemma}\label{lm:noncancellation}
Let $(a,f)\in X$, and suppose that
Assumption~\ref{ass:essential_noncancel} holds for $B_p$.
Then there is a nonempty relatively open subset of $\mcT_p$ consisting of simple tangencies at which the non-cancellation condition \eqref{eq:noncancellation:condtion} holds.
\end{lemma}
\begin{proof}
We show that the short chord created near the common tangent makes
the tangential trace nonconstant. Since the attenuation jump is
nonzero, the non-cancellation coefficient then cannot vanish
identically.

Let $m$ and $\phi_*$ be as in Assumption~ \ref{ass:essential_noncancel}, and let $\tau(\phi)$ be as in \eqref{eq:tangent:para}. In this proof, write $\rho_q(\phi)$, $\sigma_q(\phi)$, $d_q(\phi)$, and $\ell_q(\phi)$ for the ray quantities of Subsection~\ref{ssec:chords} evaluated at $(y_{\tau(\phi)},\phi+\pi/2)$. Both the point and the direction vary with $\phi$. Since $\theta^\perp(\phi+\pi/2)=-\theta(\phi)$, direct computation gives
\[
    \rho_q(\phi)=r-(c_q-c_p)\cdot\theta(\phi),
    \qquad
    \sigma_q(\phi)=-(c_q-c_p)\cdot\theta^\perp(\phi),\qquad
\rho_q'(\phi)=\sigma_q(\phi).
\]
Assumption~\ref{ass:ess_non:A2} gives $\rho_m(\phi_*)=\pm r$ and $\rho_m'(\phi_*)=\sigma_m(\phi_*)>0$. A first-order Taylor expansion gives
\begin{equation}\label{eq:rho_p,m_1}
\rho_m(\phi)^2-r^2=2\rho_m(\phi_*)\,\sigma_m(\phi_*)\,(\phi-\phi_*)+O\bigl((\phi-\phi_*)^2\bigr). 
\end{equation}
Hence, the tangent lines $L(\tau(\phi))$ enter $B_m$ on exactly one side of $\phi_*$. If $\rho_m(\phi_*)=r$, the crossing occurs for $\phi<\phi_*$.
If $\rho_m(\phi_*)=-r$, the crossing occurs for $\phi>\phi_*$.

Set $S_0:=\mcC(y_{\tau_*},\omega_*)$. Fix $q\notin\{p,m\}$. If $|\rho_q(\phi_*)|>r$, the line $L(\tau_*)$ misses $\overline B_q$. If $|\rho_q(\phi_*)|<r$, the line $L(\tau_*)$ crosses $B_q$ transversely.
In the latter case, Remark~\ref{rm:chord_analytic} gives $\sigma_q(\phi_*)\ne0$. When $\sigma_q(\phi_*)>0$, the chord lies on the backward side of $y_{\tau_*}$. When $\sigma_q(\phi_*)<0$, the chord lies on the forward side of $y_{\tau_*}$. By continuity, each of these geometric relations persists for $\phi$ near $\phi_*$. Since there are finitely many rods, we obtain
$\mcC(y_{\tau(\phi)},\phi+\pi/2)\setminus\{m\}=S_0$ throughout a neighborhood of $\phi_*$. Choose a sufficiently small one-sided interval $I$ with endpoint $\phi_*$ such that
\[
    |\rho_m(\phi)|<r,\quad \sigma_m(\phi)>0,\quad
   \mcC(y_{\tau(\phi)},\phi+\tfrac{\pi}{2})=S_0\cup\{m\},\qquad \phi\in I.
\]
All tangencies $\tau(\phi)$ with $\phi\in I$ are simple, because every rod other than $B_p$ is either missed or crossed transversely. By Assumption~\ref{ass:essential_noncancel}\ref{ass:ess_non:A3} and continuity, we may shorten $I$ so that $B_m$ is the last rod encountered along each backward ray.
For every $q\in S_0$, the ray $\mcR(y_{\tau(\phi)},\phi+\pi/2)$ crosses $B_q$ transversely on both sides of $\phi_*$. Remark~\ref{rm:chord_analytic} therefore gives real-analytic entry parameters and chord lengths near $\phi_*$. After decreasing this neighborhood if necessary, the order of the rods in $S_0$ is unchanged.

In this proof, $g\sim h$ means $g/h\to1$ as $\phi\to\phi_*$ within $I$. For $B_m$, \eqref{eq:rho_p,m_1} gives, for some constant $C>0$,
\begin{equation}\label{eq:ell_asy}
    \ell_m(\phi)
    =
    2\sqrt{r^2-\rho_m(\phi)^2}
    \sim
    C|\phi-\phi_*|^{1/2}.
\end{equation}
Using the exact derivative
\[
    \ell_m'(\phi)
    =
    -\frac{
        2\rho_m(\phi)\rho_m'(\phi)
    }{
        \sqrt{r^2-\rho_m(\phi)^2}
    },
\]
we also obtain
\begin{equation}\label{eq:ell_derivative_asy}
    |\ell_m'(\phi)|
    \sim
    \frac{C}{2}|\phi-\phi_*|^{-1/2}
    \longrightarrow\infty
\end{equation}
as $\phi\to\phi_*$ within $I$.

Let $\beta(\phi)$ be the attenuation accumulated along the segment from $y_{\tau(\phi)}$ to $y_{\tau(\phi)}-d_m(\phi)\theta_{\tau(\phi)}$.
This is $\beta_{m,p}$ from Subsection~\ref{ssec:utan}, evaluated at $\tau(\phi)$.
The segment lies in $\Omega$ by convexity.
Since $B_m$ is the last rod encountered along the backward ray, the segment contains the full chords through the rods in $S_0$.
Splitting the attenuation integral into moderator and rod contributions, we obtain
\begin{equation}\label{eq:beta_sum}
    \beta(\phi)
    =
    a_0d_m(\phi)
    +
    \sum_{q\in S_0}(a_q-a_0)\ell_q(\phi).
\end{equation}
The ray remains tangent to $B_p$, so there is no chord contribution from $B_p$ in this expression. The entry parameter is related to the chord length by
\[
    d_m(\phi)
    =
    \sigma_m(\phi)-\frac{\ell_m(\phi)}{2}.
\]
The derivatives of $\sigma_m$ and of the chord lengths
for $q\in S_0$ are bounded. Consequently, we obtain
\begin{equation}\label{eq:beta':asy}
    \beta'(\phi)
    =
    O\bigl(|\phi-\phi_*|^{-1/2}\bigr).
\end{equation}
Equation~\eqref{eq:beta_sum} also shows that $\beta$
remains bounded as $\phi\to\phi_*$.

Let $U(\phi)$ be the sum of the source contributions from the rods in $S_0$. In the notation of Subsection~\ref{ssec:utan}, this is the sum of $W_{q,p}e^{-\beta_{q,p}}$ over $q\in S_0$, evaluated at $\tau(\phi)$. Separating these contributions from that of the short chord, we obtain from Corollary~\ref{cor:u_tan_value}
\begin{equation}\label{eq:utan_local_sum}
    \utan[a,f](y_{\tau(\phi)},\theta_{\tau(\phi)})
    =
    U(\phi)
    +
    \frac{f_m}{a_m}
    \bigl(1-e^{-a_m\ell_m(\phi)}\bigr)e^{-\beta(\phi)},
    \qquad \phi\in I.
\end{equation}
The contributions to $U$ come from transversely crossed rods encountered before $B_m$. The corresponding entry parameters, chord lengths and accumulated attenuations have bounded derivatives. The sum $U$ therefore satisfies
\begin{equation}\label{eq:other_term_bounded}
    U'(\phi)=O(1).
\end{equation}
Differentiating the contribution from $B_m$ gives
\[
    \begin{aligned}
        &\frac{d}{d\phi}
        \left[
            \frac{f_m}{a_m}
            \bigl(1-e^{-a_m\ell_m(\phi)}\bigr)
            e^{-\beta(\phi)}
        \right]\\
        &\qquad=
        e^{-\beta(\phi)}
        \left[
            f_m e^{-a_m\ell_m(\phi)}\ell_m'(\phi)
            -
            \frac{f_m}{a_m}
            \bigl(1-e^{-a_m\ell_m(\phi)}\bigr)
            \beta'(\phi)
        \right].
    \end{aligned}
\]
By \eqref{eq:ell_asy}, we have
\[
    1-e^{-a_m\ell_m(\phi)}=O\bigl(|\phi-\phi_*|^{1/2}\bigr).
\]
Together with \eqref{eq:beta':asy}, this shows that the second term in brackets is $O(1)$.
The absolute value of the first term tends to infinity by \eqref{eq:ell_derivative_asy} and $f_m\ne0$.
Since $e^{-\beta(\phi)}$ remains bounded away from zero, \eqref{eq:utan_local_sum} and \eqref{eq:other_term_bounded} imply
\begin{equation}\label{eq:utan_derivative_blowup}
    \left|
        \frac{d}{d\phi}
        \utan[a,f](y_{\tau(\phi)},\theta_{\tau(\phi)})
    \right|
    \longrightarrow\infty.
\end{equation}

Thus, the tangential trace is nonconstant on $I$.
Since $a_p-a_0\ne0$, the function $\phi\mapsto F_p^{(a,f)}(\tau(\phi))$ cannot vanish identically on $I$.
By continuity, there is a nonempty open subinterval $J\subset I$ on which $F_p^{(a,f)}(\tau(\phi))\ne0$.
The map $\phi\mapsto\tau(\phi)$ sends $J$ to a nonempty relatively open subset of $\mcT_p$ consisting of simple tangencies satisfying \eqref{eq:noncancellation:condtion}.
\end{proof}
\begin{remark}\label{rm:noncancel_sufficient}
\begin{enumerate}[label=(\alph*)]
    \item Assumption~\ref{ass:essential_noncancel} is sufficient, not necessary. Suppose that the backward ray at a simple tangency $\tau$ of $B_p$ meets no emitting rod. Then $\utan[a,f](y_\tau,\theta_\tau)=0$.
Consequently, $F_p^{(a,f)}(\tau)=f_p$. If $f_p>0$, the non-cancellation condition holds even when $a_p=a_0$.
For $H(\omega)=R_af(x_\tau,\theta(\omega))$, we have $\Phi_D[H]\ne0$. If also $a_p=a_0$, then $\Phi_N[H]=0$. The quotient \eqref{eq:Delta_a:ND} therefore returns $a_p=a_0$. The proof of Theorem~\ref{tm:recovery} requires only one non-cancelling simple tangency per rod. The reconstruction argument applies whenever such tangencies are available.
Under the row-separation assumption of Section~\ref{sec:stability} and $f_{\min}>0$, Proposition~\ref{prop:existence_admissible_family} provides an admissible tangency for every rod.
\item Lemma~\ref{lm:noncancellation} allows the target rod $B_p$ to be non-emitting. For $f_p=0$ and $a_p\ne a_0$, an emitting rod $B_m$ satisfying Assumption~\ref{ass:essential_noncancel} provides the variation of the tangential trace used in the proof.
\item If $a_p=a_0$ and $f_p=0$, the coefficients in $B_p$ agree with those in the moderator. At every simple tangency of $B_p$, \eqref{eq:PhiD_H}--\eqref{eq:PhiN_H} therefore give $\Phi_D[H]=\Phi_N[H]=0$. The quotient formula does not recover this rod's coefficients from these singular features.
\item The attenuation reconstruction uses angular limits of the data. It does not assert recovery from finitely many point samples.
\end{enumerate} 
\end{remark}

\begin{theorem}[Unique recovery]\label{tm:recovery}
Let $(a,f)\in X$, and suppose that
Assumption~\ref{ass:essential_noncancel} holds for every rod
$B_p$, $p=1,\dots,K$. Then $(a,f)$ can be reconstructed uniquely from the full data
\[
    \bigl\{
        R_af(\psi,\omega):
        (\psi,\omega)\in[0,2\pi)^2
    \bigr\}.
\]
\end{theorem}

\begin{proof}
Fix $p \in \{1,\dots,K\}$. By Lemma~\ref{lm:noncancellation}, there exists $\tau=(\psi,\omega_0)\in\mcS_p$ at which
$F_p^{(a,f)}(\tau)\ne0$. For $H(\omega):=R_af(\psi,\omega)$,
Proposition~\ref{prop:HR} gives $\Phi_D[H]\ne0$, and we obtain
\[
    a_p=a_0+
    \frac{\Phi_N[H]}{\sqrt{2rL_\tau}\,\Phi_D[H]}.
\]
The geometry is known, and the data determine both functionals,
so this formula reconstructs $a_p$. The condition $\Phi_D[H]\ne0$
identifies the required tangencies from the data. Repeating the
argument for every rod recovers $a$.

Once $a$ is known, we recover $f$ by Novikov's inversion formula. The zero extensions of $a$ and $f$ are bounded and compactly supported. Hence, both extensions belong to $L^{\infty,1+\varepsilon}(\mathbb R^2,\mathbb R)$ as defined in \cite[(1.9)--(1.10)]{Novikov:2002}, for every $\varepsilon>0$. The full boundary-angle data determine the transform on all oriented lines meeting $\Omega$. The transform vanishes on every line missing $\Omega$. Represent each oriented line by the point on the line closest to the origin and the direction of the line. With this representation, the data coincide with $P_af$ in \cite[(1.1)--(1.2)]{Novikov:2002}. The inversion formula, with the distributional interpretation in
\cite[Theorem~2.1 and Remark~2.6]{Novikov:2002}, therefore recovers $f$ almost everywhere.
\end{proof}
Under the assumptions of Section~\ref{sec:stability}, Remark~\ref{remark:reconstruction-procedure} gives explicit formulas for both coefficients and a self-contained proof of uniqueness in that setting.
\section{Singular stability estimates}\label{sec:stability}

For stability, we use tangencies at which the backward ray meets
no other rod. The tangential trace then vanishes, and a positive
lower bound on the rod sources gives a uniform non-cancellation
bound.  
\begin{definition}[Admissible tangency]\label{def:admissible}
A simple tangency $\tau=(\psi,\omega_0)\in\mcS_m$ is called \emph{admissible} if the backward ray $\mcR(y_\tau,\omega_0)$ meets no closed rod other than $\overline{B_m}$, that is, if
\begin{equation}\label{eq:clear_tangency}
    \mcR(y_\tau,\omega_0)\cap\overline{B_q}=\emptyset,
    \qquad q\ne m.
\end{equation}
\end{definition}
\begin{assumption}\label{ass:stability}
\begin{enumerate}[label=\textup{(A\arabic*)},ref=(A\arabic*),leftmargin=3em]
    \item \label{ass:stab:X} The geometry $\mcG=(\ro,r,\{c_q\}_{q=1}^K)$ of \eqref{eq:geometry} and the parameter space $X=X(a_{\min},a_{\max},f_{\min},f_{\max},a_0,\mcG)$ of Definition~\ref{def:parameter space} are fixed.
  \item\label{ass:stab:fmin} \emph{Emitting rods.}  We assume $f_{\min}>0$. Thus, every rod source is bounded below by the same positive constant throughout $X$.
   \item\label{ass:stab:rows} \emph{Row separation.} The rod centers lie on parallel lines, called  \emph{rows}. Any two distinct rows are separated by more than one rod diameter. Equivalently, there is a unit vector $n\in\mathbb S^1$ such that, for every $p,q\in \{1,\cdots, K\}$, we have
    \begin{equation}\label{eq:row_separation}
      (c_q-c_p)\cdot n=0  
        \qquad\text{or }\qquad
       |(c_q-c_p)\cdot n|>2r .
    \end{equation}
  The vector $n$ is a common unit normal to the rows.
    \item\label{ass:stab:selection} \emph{Selected tangencies.} For every $m\in\{1,\dots,K\}$, an admissible tangency $\tau_m=(\psi_m,\omega_m)\in\mcS_m$ is fixed.
    \item\label{ass:stab:data} \emph{Data.} The measurement set $\mcM$ of Definition~\ref{def:data} contains the angular window of every selected tangency, so that
    \[
        \{\psi_m\}\times(\omega_m-\delta_{\tau_m},\omega_m+\delta_{\tau_m})\subset\mcM,
        \qquad m=1,\dots,K,
    \]
    that is, \eqref{eq:measurement_windows} holds for $\tau=\tau_m$.
    \end{enumerate}
\end{assumption}
Proposition~\ref{prop:existence_admissible_family} below shows that \ref{ass:stab:rows} guarantees an admissible tangency for every rod, as required by \ref{ass:stab:selection}. The tangencies $\tau_1,\dots,\tau_K$ may be chosen freely among the admissible tangencies. The constants in Theorem~\ref{tm:stability} are independent of this choice. Assumption~\ref{ass:stab:data} requires data at most $K$ boundary positions, with an interval of directions at each position.

The lower bound \ref{ass:stab:fmin} makes the estimates uniform over
$X$. At an admissible tangency, the tangential trace vanishes, so
the non-cancellation coefficient equals $f_m\ge f_{\min}$ and the
denominator in \eqref{eq:Delta_a:ND} is bounded away from zero. See Proposition~\ref{prop:existence_admissible_family}. The uniform
estimate therefore excludes non-emitting rods. For uniqueness
without this restriction, see Remark~\ref{rm:noncancel_sufficient}.
\begin{remark}\label{rm:clear_simple_tangency}
The rectangular and hexagonal arrangements of fuel assemblies motivate \ref{ass:stab:rows} as in \cite[Fig.~3.2]{Virta:2024}. The row-separation condition \eqref{eq:row_separation} must be checked using the center spacing and rod diameter, rather than inferred from the arrangement alone.
\end{remark}
\begin{proposition}[Existence of admissible tangencies]
\label{prop:existence_admissible_family}
Suppose that assumptions~\ref{ass:stab:X} and \ref{ass:stab:rows} hold. Then every rod $B_m$ has an admissible tangency. At every admissible tangency $\tau\in\mcS_m$, the non-cancellation coefficient reduces to
\[
    F_m^{(a,f)}(\tau)=f_m
    \qquad\text{for all }(a,f)\in X.
\]
In particular, under \ref{ass:stab:fmin}, $F_m^{(a,f)}(\tau)\ge f_{\min}>0$ uniformly in $(a,f)\in X$.
\end{proposition}
\begin{proof}
We construct the tangent direction by tilting it slightly from the row direction, as illustrated in
Figure~\ref{fig:Backward-rays}. Let $n$ be the row normal of \ref{ass:stab:rows}, and let $e$ be the unit vector whose counterclockwise rotation through $\pi/2$ is $n$. Since there are finitely many rows, choose $d>2r$ such that
every pair of distinct rows is separated by at least $d$. We set
\[
   w_e:=\max_{p,q}|(c_q-c_p)\cdot e|.
\]
Choose $\alpha\in(0,\pi/2)$ sufficiently small that
\[
    d\cos\alpha-w_e\sin\alpha>2r,
    \qquad
  w_e\sin\alpha<2r.
\]
This is possible by continuity at $\alpha=0$. We define
\[
    \vartheta:=\cos\alpha\,e+\sin\alpha\,n,
    \qquad
    \varpi:=-\sin\alpha\,e+\cos\alpha\,n.
\]
Fix $m \in \{1,\dots,K\}$ and set $y:=c_m-r\varpi$. The line $L(y,\vartheta)$ is tangent to $\partial B_m$. For $q\ne m$, write $c_q-c_m=\xi e+\eta n$. The distance from $c_q$ to $L(y,\vartheta)$ is
\[
    |(c_q-y)\cdot\varpi|=|r-\xi\sin\alpha+\eta\cos\alpha|.
\]
If this distance is greater than $r$, then $L(y,\vartheta)$ misses $\overline{B_q}$. If this distance is less than $r$, then $L(y,\vartheta)$ intersects the interior of $B_q$.

If $B_q$ belongs to a different row, then $|\eta|\ge d$, and we obtain
\begin{align*}
    |r-\xi\sin\alpha+\eta\cos\alpha|
    &\ge
    |\eta|\cos\alpha-|r-\xi\sin\alpha|\ge
    d\cos\alpha-r-|\xi|\sin\alpha\\&\ge
    d\cos\alpha-r-w_e\sin\alpha>r.
\end{align*}
Thus, $L(y,\vartheta)$ misses $\overline{B_q}$.

If $B_q$ belongs to the same row, then $\eta=0$ and $\xi\ne0$. When $\xi<0$, we have
\[
    |r-\xi\sin\alpha|
    =
    r+|\xi|\sin\alpha
    >
    r.
\]
Thus, $L(y,\vartheta)$ again misses $\overline{B_q}$. When $\xi>0$, we have
\[
      0<\xi\sin\alpha\le w_e\sin\alpha<2r,
    \qquad
    |r-\xi\sin\alpha|<r.
\]
The line $L(y,\vartheta)$ crosses $B_q$ transversely. Moreover, we have
\[
    (y-c_q)\cdot\vartheta=-\xi\cos\alpha<0.
\]
Since the closed rods are pairwise disjoint, we have $|y-c_q|>r$. For every $s\ge0$, we therefore obtain
\[
    |y-s\vartheta-c_q|^2
    =|y-c_q|^2-2s(y-c_q)\cdot\vartheta+s^2>r^2.
\]
Hence, the backward ray $\{y-s\vartheta:s\ge0\}$ misses $\overline{B_q}$. The chord $L(y,\vartheta)\cap B_q$ therefore lies entirely on the forward side of $y$.

These cases cover every $q\ne m$. Thus, $L(y,\vartheta)$ is tangent to no rod boundary other than $\partial B_m$. The backward ray $\{y-s\vartheta:s\ge0\}$ meets no other closed rod. Choose the entry point of the oriented line $L(y,\vartheta)$ into $\Omega$ as the boundary point. This gives a simple tangency $\tau$ with $y_\tau=y$ and $\theta_\tau=\vartheta$ satisfying \eqref{eq:clear_tangency}. 

For any admissible tangency $\tau=(\psi,\omega_0)\in\mcS_m$, condition \eqref{eq:clear_tangency} gives $\mcC(y_\tau,\omega_0)=\emptyset$. Corollary~\ref{cor:u_tan_value} therefore yields for any $(a,f)\in X$ that
\[
    \utan[a,f](y_\tau,\theta_\tau)=0.
\]
Consequently, $F_m^{(a,f)}(\tau)=f_m$. Under \ref{ass:stab:fmin}, we obtain $F_m^{(a,f)}(\tau)\ge f_{\min}>0$ uniformly in $(a,f)\in X$.
\end{proof}

\begin{figure}
    \centering
    \begin{tikzpicture}
    \begin{scope}[scale=1.5]
        \path[clip] (-10,3.1) rectangle (-2,-0.4); 

        \coordinate (c) at (-0.9, -1.05);
        \draw[gray, dashed, fill=gray!20!white, fill opacity=0.5] ($(c) + (-10cm, 26mm)$) rectangle +(10cm, 8mm); 
        \foreach \i in {-5, ..., -1}{
            \foreach \j in {0, ..., 3}{
                \def\ar{3}
                \def\a{\fpeval{\ar+90}}
                \coordinate (xm) at ($(c) + 2*(\i + 0.5*\j,\j)$); 
                \coordinate (xstar) at ($(xm) + (\a:6mm)$); 
                \draw[black, fill=blue, opacity=0.5] (xm) circle (6mm);
                \if\j1
                    \draw[red, opacity=.5] (xstar) -- +(\ar:20cm); 
                    \fill[red, opacity=.7] (xstar) circle (1pt) node[above] {};
                \fi
            }
        }

    \end{scope}
    \end{tikzpicture}
    \caption{Admissible tangencies obtained by tilting the tangent lines relative to the row direction.
    Each line is tangent to exactly one rod. The backward ray from each tangency point meets no other rod.}
    \label{fig:Backward-rays}
\end{figure}
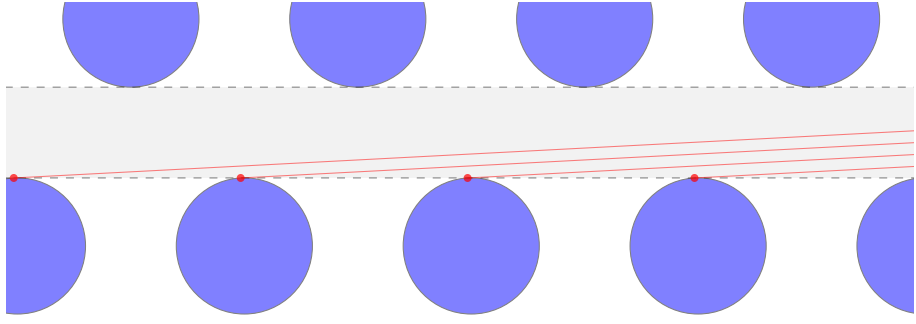
\subsection{The stability seminorm and the main estimate}\label{ssec:stability_estimate}
Let Assumption~\ref{ass:stability} hold. Since the selection $\tau_1,\dots,\tau_K$ is finite, we have
\begin{equation}\label{eq:delta_star}
    \min_{1\le m\le K}\delta_{\tau_m}>0,
\end{equation}
and Lemma~\ref{lm:L_bounds} gives
\begin{equation}\label{eq:L_global_bounds}
    0<L_{\min}\le L_{\tau_m}\le L_{\max}<\infty,\qquad m=1,\dots,K.
\end{equation}
\begin{definition}[Stability seminorm]\label{def:stab_seminorm}
Let Assumption~\ref{ass:stability} hold. For a real-valued
function $H$ on $\mcM$ such that
$H(\psi_m,\cdot)|_{I_{\tau_m}^\circ}\in\mcV_{\tau_m}$
for every $m$, we define
\begin{equation}\label{eq:data_stability_seminorm}
    \|H\|_{\star}
    :=\max_{1\le m\le K}
      \|H(\psi_m,\cdot)\|_{\tau_m}.
\end{equation}
\end{definition}
These functions form a vector space on which the finite maximum
of the local seminorms is a seminorm. For two
pairs in $X$, Proposition~\ref{prop:HR} and the vector space structure of $\mcV_{\tau_m}$ show that the data difference belongs to this space. Assumption~\ref{ass:stab:data} ensures that the required restrictions are available.
\begin{remark}\label{remark:reconstruction-procedure}
Under Assumption~\ref{ass:stability}, the selected tangencies also give explicit reconstruction formulas. For each $m$, write $H_m:=R_af(\psi_m,\cdot)$. By Proposition~\ref{prop:existence_admissible_family}, one has $F_m^{(a,f)}(\tau_m)=f_m\ge f_{\min}>0$. Hence, Proposition~\ref{prop:HR} gives
\[
    a_m
    =a_0+
      \frac{\Phi_N[H_m]}
           {\sqrt{2rL_{\tau_m}}\,\Phi_D[H_m]},
    \qquad
    f_m
    =\frac{e^{Da(y_{\tau_m},\theta_{\tau_m})}}
           {\sqrt{2rL_{\tau_m}}}\,\Phi_D[H_m].
\]
First recover all attenuation coefficients using the first formula.
The beam transform in the second formula is then computable from
the known geometry and the recovered attenuation, so this formula
recovers every source coefficient.
Thus, $2K$ singular functionals determine $(a,f)$ in this setting.
These functionals require angular limits, and they are not individual
point measurements.
\end{remark}

\begin{theorem}\label{tm:stability}
Let Assumption~\ref{ass:stability} hold. Then there exist constants $\mathscr C_a,\mathscr C_f>0$, depending only on $X$ and $\mcG$, and not on the selected tangencies $\tau_1,\dots,\tau_K$, such that
\begin{equation}\label{eq:a_stability}
    \|a^{(1)}-a^{(2)}\|_{\ell^\infty}
    \le
    \mathscr C_a
    \|R_{a^{(1)}}f^{(1)}-R_{a^{(2)}}f^{(2)}\|_{\star}
\end{equation}
and
\[
    \|f^{(1)}-f^{(2)}\|_{\ell^\infty}
    \le
    \mathscr C_f
    \|R_{a^{(1)}}f^{(1)}-R_{a^{(2)}}f^{(2)}\|_{\star}
\]
for every $(a^{(1)},f^{(1)}),(a^{(2)},f^{(2)})\in X$.
\end{theorem}

\begin{remark}The proof gives explicit formulas for $\mathscr C_a$ and $\mathscr C_f$ in \eqref{eq:C_a_constant} and \eqref{eq:C_f_constant}, respectively.
\end{remark}

\begin{proof}
We first estimate the attenuation using
\eqref{eq:PhiN_H}--\eqref{eq:Delta_a:ND}, and then the
source using \eqref{eq:PhiD_H}. Fix $m\in\{1,\dots,K\}$, and write $H:=H^{(1)}-H^{(2)}$, where
\[
    H^{(i)}:=R_{a^{(i)}}f^{(i)}(\psi_m,\cdot),
    \qquad i=1,2.
\]
The bounds \eqref{eq:L_global_bounds} hold for $L_{\tau_m}$.
We set
\begin{equation}\label{eq:A_delta}
    A_\Delta
    :=
    \max\{|a_{\min}-a_0|,|a_{\max}-a_0|\},
    \quad
    A:=\max\{a_0,a_{\max}\},\quad b_i:=Da^{(i)}(y_{\tau_m},\theta_{\tau_m}).
\end{equation}
Then $|a_m^{(i)}-a_0|\le A_\Delta$. Since a ray travels at most
$2\ro$ inside $\Omega$, we also have
\begin{equation}\label{eq:beam_uniform_bound}
    0\le
  b_i
    \le 2\ro A,
    \qquad i=1,2.
\end{equation}
Proposition~\ref{prop:existence_admissible_family} gives $F_m^{(a^{(i)},f^{(i)})}(\tau_m)=f_m^{(i)}\ge f_{\min}$.
Using \eqref{eq:PhiD_H}, \eqref{eq:beam_uniform_bound} and \eqref{eq:L_global_bounds}, we obtain
\begin{equation}\label{eq:BD}
    \Phi_D[H^{(i)}]=\sqrt{2rL_{\tau_m}}\,f_m^{(i)}e^{-b_i},\quad
\bigl|\Phi_D[H^{(i)}]\bigr|\ge\sqrt{2rL_{\min}}\,f_{\min}\,e^{-2\ro A}=:\mathcal B_D>0.
\end{equation}
By \eqref{eq:PhiN_H}, we have \[
    \Phi_N[H^{(i)}]=(a_m^{(i)}-a_0)\sqrt{2rL_{\tau_m}}\,\Phi_D[H^{(i)}].
\] Subtracting the two identities, we obtain
\begin{multline*}
    \bigl(a_m^{(1)}-a_m^{(2)}\bigr)\sqrt{2rL_{\tau_m}}\,\Phi_D[H^{(1)}]\\
    =\bigl(\Phi_N[H^{(1)}]-\Phi_N[H^{(2)}]\bigr)-(a_m^{(2)}-a_0)\sqrt{2rL_{\tau_m}}\,\bigl(\Phi_D[H^{(1)}]-\Phi_D[H^{(2)}]\bigr).
\end{multline*}
By linearity of the functionals at $\tau_m$, we have\[
\begin{aligned}
    \Phi_N[H^{(1)}]-\Phi_N[H^{(2)}]&=\Phi_N[H],\\
    \Phi_D[H^{(1)}]-\Phi_D[H^{(2)}]&=\Phi_D[H].
\end{aligned}
\] Consequently, we obtain
\[
    \bigl|a_m^{(1)}-a_m^{(2)}\bigr|
    \le\frac{|\Phi_N[H]|+A_\Delta\sqrt{2rL_{\max}}\,|\Phi_D[H]|}{\sqrt{2rL_{\min}}\,\mathcal B_D}
    \le\frac{\max\{1,A_\Delta\sqrt{2rL_{\max}}\}}{2rL_{\min}f_{\min}e^{-2\ro A}}\,\|H\|_{\tau_m}.
\]
Taking the maximum over $m$ proves \eqref{eq:a_stability} with
\begin{equation}\label{eq:C_a_constant}
    \mathscr C_a:=\frac{e^{2\ro A}}{2rL_{\min}f_{\min}}\max\{1,A_\Delta\sqrt{2rL_{\max}}\}.
\end{equation}
To estimate the source, recall that $$\Phi_D[H^{(i)}]=\sqrt{2rL_{\tau_m}}f_m^{(i)}e^{-b_i}.$$ Since the estimate is symmetric in the two pairs, we may assume $b_1\le b_2$. Rearranging the difference of these identities, we obtain
\[
    f_m^{(1)}-f_m^{(2)}
    =\frac{e^{b_1}}{\sqrt{2rL_{\tau_m}}}\,\Phi_D[H]+f_m^{(2)}\bigl(e^{b_1-b_2}-1\bigr).
\]
Since $1-e^{-t}\le t$ for $t\ge0$, it follows that
\[
    |f_m^{(1)}-f_m^{(2)}|
    \le\frac{e^{2\ro A}}{\sqrt{2rL_{\min}}}\,|\Phi_D[H]|+f_{\max}|b_1-b_2|.
\]

The moderator attenuation is the same for both pairs. Therefore, we obtain
\begin{align*}
    |b_1-b_2|
    &\le
    \int_0^\infty
        \left|
            a^{(1)}(y_{\tau_m}+s\theta_{\tau_m})
            -
            a^{(2)}(y_{\tau_m}+s\theta_{\tau_m})
        \right|\,ds\\
    &\le
    2\ro\,
    \|a^{(1)}-a^{(2)}\|_{\ell^\infty}.
\end{align*}
Combining this bound with Remark~\ref{rm:seminorm} and
\eqref{eq:a_stability}, and taking the maximum over $m$, proves
the source estimate with
\begin{equation}\label{eq:C_f_constant}
       \mathscr C_f
    :=
    \frac{e^{2\ro A}}{\sqrt{2rL_{\min}}}
    +2\ro f_{\max}\mathscr C_a.
\end{equation}
Both constants depend only on the global geometric and coefficient
bounds and are therefore independent of the selected tangencies.
\end{proof}

\section{Numerical experiment} \label{sec:numerics}

To complement the stability estimates, we conduct numerical experiments to examine the asymptotic behavior of the singular functionals.
 
By Definition~\ref{def:functionals} and \eqref{eq:data_stability_seminorm}, we have
\begin{equation*}
    \| R_{a}f - R_{\tilde{a}}\tilde{f} \|_{\star} = \max_{1 \leq m \leq K} \|  R_{a}f (\psi_m, \cdot ) - R_{\tilde{a}}\tilde{f} (\psi_m, \cdot ) \|_{\tau_m},
\end{equation*}
for simple tangencies $\tau_m \in \mcS_m$ chosen to satisfy \eqref{eq:clear_tangency}. Here we consider the case where $(a,f)$ and $(\tilde{a},\tilde{f})$ differ only at a single rod $B_p$ and focus on its simple tangencies. With such a simple phantom, we can easily define the rod locations and radii, and then compute exact projections.

Given the chosen parameters, we can estimate the differences $\Phi_D[H]$ and $\Phi_N[H]$ of \eqref{eq:PhiD-def} and \eqref{eq:PhiN-def} by analytically computing the ray and chord lengths for the given ray angle and fixed boundary point $x(\psi)$. However, the angular partial derivatives of the data $\partial_\omega R_a f \left( \psi_p, \omega \right)$ need to be estimated with finite differences, which makes taking the limit $\omega \to \omega_0^{\etan}$ numerically unstable. Hence, we denote $\gamma^2 = | \omega - \omega_0 |,$ where the rays again approach from inside the disk for the quantities to be meaningful. Notice that by Lemma~\ref{lm:side}, we have $|\omega - \omega_0| = \etan (\omega - \omega_0)$ when taking the one-sided limits $\omega \to \omega_0^{\etan}$.

Now for $\gamma > 0$, we can define an auxiliary function $\varphi_+ (\gamma) := H(\omega_0 + \etan \gamma^2)$ and by differentiating, we get
\begin{align}
    \varphi_+'(\gamma) &= 2 \etan \gamma H'(\omega_0 + \etan \gamma^2) \\
    \Rightarrow \ \lim_{\gamma \to 0^+} \frac{1}{2} \varphi_+'(\gamma) &= \lim_{\omega \to \omega_0^{\epsilon_\tau}} |\omega - \omega_0|^{\frac{1}{2}}  H' (\omega_0 + (\omega - \omega_0)) = \Phi_D[H], \label{eq:stablePhiD}
\end{align}
where we estimate $\varphi_+'$ with forward difference of fixed step size $h$ by
\begin{align}
    \frac{1}{2} \varphi_+'(\gamma) &\approx \frac{1}{2} \frac{\varphi_+(\gamma + h) - \varphi_+(\gamma)}{h} = \frac{1}{2} \frac{H \big( \omega_0 + \etan (\gamma + h)^2 \big) - H\big( \omega_0 + \etan \gamma^2 \big)}{h}. \label{eq:num_lim_PhiD}
\end{align}

Similarly, we can define a second auxiliary function $\varphi_{-} (\gamma) := H(\omega_0 - \etan \gamma^2)$ for the reflected direction, and differentiate both of them to get
\begin{align*}
    \varphi_{+}''(\gamma) &= 4 \gamma^2 H''(\omega_0 + \etan \gamma^2) + 2 \etan H'(\omega_0 + \etan \gamma^2), \\
    \varphi_{-}'(\gamma) &= - 2 \etan \gamma H'(\omega_0 - \etan \gamma^2).
\end{align*}
Then
\begin{align*}
    &\lim_{\gamma \to 0^+}  \frac{1}{4} \varphi_{+}'' (\gamma) + \frac{1}{4 \gamma} \varphi_{-}'(\gamma)  \\
    = &\lim_{\omega \to \omega_0^\etan}  \gamma^2 H''(\omega_0 + (\omega - \omega_0)) + \frac{\etan}{2} H'(\omega_0 + (\omega - \omega_0)) - \frac{\etan}{2} H'(\omega_0 - (\omega - \omega_0)) =  \Phi_N[H] .
\end{align*}
These we can estimate using first and second order finite differences
\begin{align}
    \frac{1}{4} \varphi_{+}'' (\gamma) + \frac{1}{4 \gamma} \varphi_{-}'(\gamma) \approx & \frac{1}{4} \frac{\varphi_{+}(\gamma + h) - 2 \varphi_{+}(\gamma) + \varphi_{+}(\gamma - h)}{h^2} + \frac{1}{4\gamma} \frac{\varphi_{-}(\gamma + h) - \varphi_{-}(\gamma)}{h} \nonumber \\
    =& \frac{1}{4} \frac{H \big( \omega_0 + \etan (\gamma + h)^2 \big) - 2 H \big( \omega_0 + \etan \gamma^2 \big) + H \big( \omega_0 + \etan (\gamma - h)^2 \big)}{h^2} \nonumber \\
    &+ \frac{1}{4\gamma} \frac{H \big( \omega_0 - \etan (\gamma + h)^2 \big) - H \big( \omega_0 - \etan \gamma^2 \big)}{h}. \label{eq:num_lim_PhiN}
\end{align}
We can also use $\varphi_-$ in place of $\varphi_+$ to estimate $\Phi_D$ for rays outside the rod. Recall that $\Phi_N$ is not defined for such rays.

With the approximations \eqref{eq:num_lim_PhiD} and \eqref{eq:num_lim_PhiN}, we fix the finite difference step length $h = 10^{-4}$ and compute the values of $\varphi_+'(\gamma)$ and $\varphi_{+}''(\gamma) + \gamma^{-1} \varphi_{-}'(\gamma)$ for different angular values on both sides of the tangential angle $\gamma_0$. We compute these estimates for three different perturbations of the rod $B_p$. One where the emission is 20\% lower, one where the attenuation is 50\% higher and one with both at the same time. While our domain radius ($\ro = 165$ mm) and attenuation values ($a_0 = 0.0085$ mm$^{-1}, a_p = 0.14$ mm$^{-1}$) are realistic, the rod diameter chosen ($r = 8.25$ mm) is roughly twice as large as in any real fuel assembly. The larger radius makes the limits slightly more stable to compute. Moreover, the phantom with at most two rods is much simpler than any real assembly, but suitable for our tests.

We also consider two different types of simple tangencies illustrated in Figure~\ref{fig:setup}: (i) a ray which is only tangent to the perturbed rod and never meets any other rods; and (ii) a ray which is tangent to the perturbed rod but first passes through another rod earlier along the ray. Note that in both cases the \emph{forward ray} from the tangency to the detector meets no other rods as required. The numerical derivative estimates of the different perturbations for the two rays are shown in Figure~\ref{fig:plot_case_i} and Figure~\ref{fig:plot_case_ii}, respectively. Proposition~\ref{prop:HR} also gives an alternative way of obtaining the limit values $\Phi_D, \Phi_N$ without differentiation. We have estimated these values and included them in the plots with dashed lines, matching the colors of the different perturbations. Note that these values are only obtained at the tangency, where $\omega = \omega_0$.

\begin{figure}[ht]
    \centering
    \begin{tikzpicture}
        \node[anchor=west] (c1) at (0,0) {\includegraphics[width=0.4\linewidth]{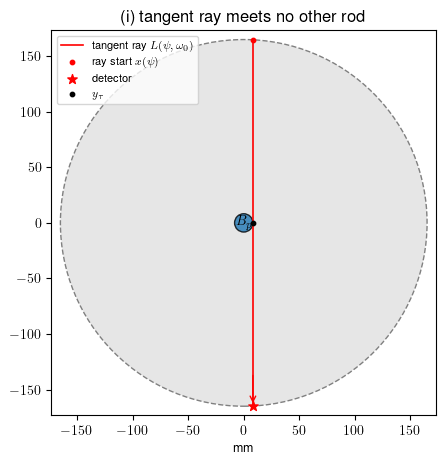}};
        
        \node[anchor=west] (c3) at (8cm,0) {\includegraphics[width=0.4\linewidth]{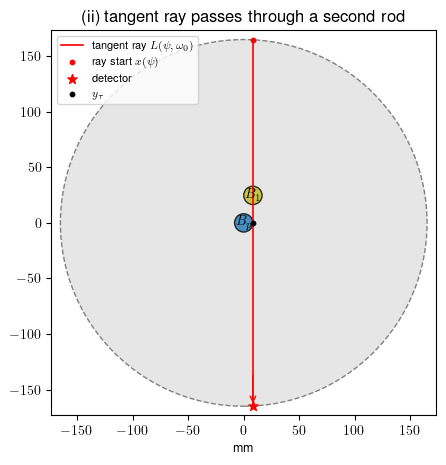}};
    \end{tikzpicture}
    
    \caption{Geometric setup illustrating two different simple tangencies. In (i), the tangent ray meets no other rods. In (ii), the tangent ray passes through a second rod.}
    \label{fig:setup}
\end{figure}

\begin{figure}[ht]
    \centering
    \includegraphics[width=1\linewidth]{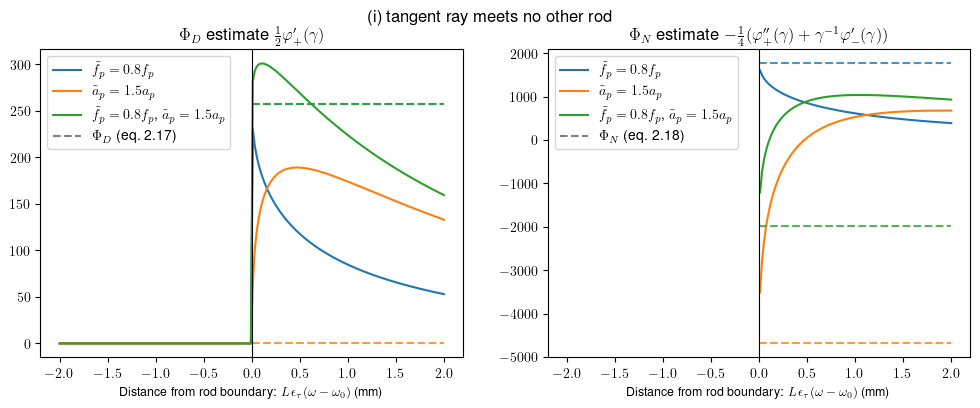}
    \caption{Numerical estimates of $\Phi_D$ and $\Phi_N$ for the ray (i) that meets no other rods. The three tested perturbations of emission and attenuation are shown in different colors.}
    \label{fig:plot_case_i}
\end{figure}

\begin{figure}[ht]
    \centering
    \includegraphics[width=1\linewidth]{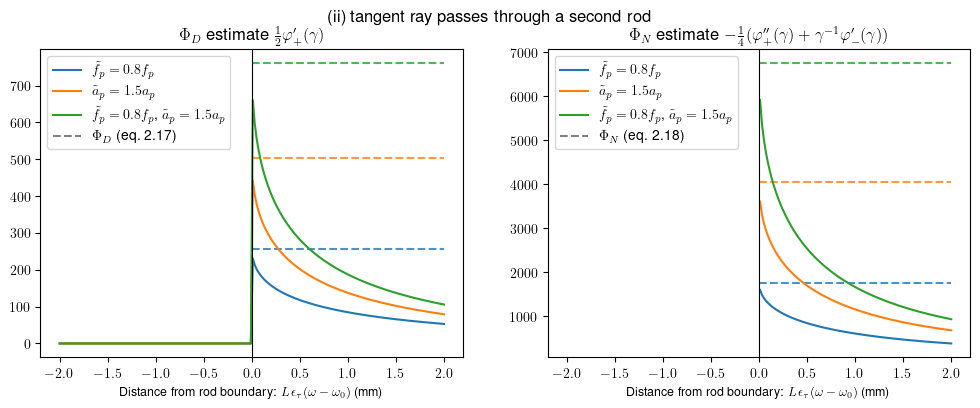}
    \caption{Numerical estimates of $\Phi_D$ and $\Phi_N$ for the ray (ii) that passes through another rod. The three tested perturbations of emission and attenuation are shown in different colors.}
    \label{fig:plot_case_ii}
\end{figure}

As can be seen from Figures~\ref{fig:plot_case_i} and \ref{fig:plot_case_ii}, the singularity at the tangential ray and vanishing data are captured well, and the angular derivative estimates tend to the values predicted by Proposition~\ref{prop:HR}. However, at that scale it is still hard to judge how accurately the limit value is obtained, in particular if the jump in attenuation values should be computed.

Figure~\ref{fig:Phi_error_ii} shows the relative difference between the estimated values of $\Phi_D[H]$ and $\Phi_N[H]$ compared to the known values, for decreasing values of $\gamma = \sqrt{| \omega - \omega_0 |}$ for the case (ii) where the tangent ray passes through a second rod. Based on this test, the desired asymptotic behavior is obtained for every tested perturbation when $\gamma < 10^{-2}$ or $|\omega - \omega_0 | < 10^{-4}$. However, at around $\gamma \approx 10^{-4}$, which is also the step size $h$ for the numerical derivatives, the estimates of $\Phi_N$ become unstable. In our tests trying smaller values of $h$ or $\gamma < h$ both lead to numerical instabilities. 
We should be aware of these limitations for the next part.

\begin{figure}
    \centering
    \includegraphics[width=1\linewidth]{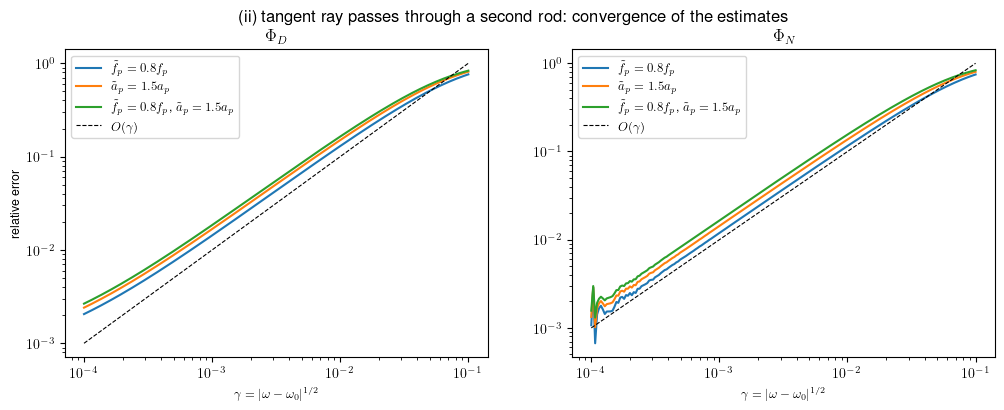}
    \caption{Plot of the relative difference between the estimated and known values of $\Phi_D$ (right) and $\Phi_N$ (left) for decreasing values of the off-tangent angle $\omega$. In addition the $\mathcal{O}(\gamma) = \mathcal{O}(\sqrt{| \omega - \omega_0|}$ asymptotics are shown with black dashed lines.}
    \label{fig:Phi_error_ii}
\end{figure}

Rather than a difference in perturbed inputs, from now on we only consider single input $(a, f)$ and denote $H(\omega) = R_a f(\psi, \omega)$. 
The clean numerical asymptotics for $\gamma \in [5 \cdot 10^{-4}, 10^{-2}]$ should be sufficient for quadratic fit to estimate the value at the limit $\gamma \to 0$. We may then find the attenuation and emission values $a_p$ and $f_p$ for the rod $B_p$, we are tangent to, using formulas from remark~\ref{remark:reconstruction-procedure} as
\begin{align*}
    a_p = a_0 + \frac{\Phi_N[H]}{\sqrt{2rL_{\tau}}\,\Phi_D[H]},
    \qquad
    f_p = \frac{e^{Da(y_{\tau},\theta_{\tau})}}{\sqrt{2rL_{\tau}}} \Phi_D[H] + (a_p - a_0) u(y_\tau, \theta_\tau).
\end{align*}

These values with the relative error are shown in Table~\ref{tab:recons} for the three different variants of $(\tilde a, \tilde f)$ and two geometric setups (i) and (ii).

\begin{table}[htbp]
    \centering
    \begin{tabular}{ccc|clcl}
                            & True $a_p$ & True $f_p$ & $a_p$ & rel. error & $f_p$ & rel. error \\ \hline
        \multirow{3}{*}{(i)}& 0.14 & 80  & 0.14009 & 0.0007 & 79.944 & 0.0007 \\
                            & 0.21 & 100 & 0.21022 & 0.0010 & 99.887 & 0.0011 \\
                            & 0.21 & 80  & 0.21022 & 0.0010 & 79.910 & 0.0011 \\ \hline
        \multirow{3}{*}{(ii)}&0.14 & 80  & 0.13992 & 0.0005 & 79.952 & 0.0006 \\
                            & 0.21 & 100 & 0.21031 & 0.0015 & 100.187& 0.0019 \\
                            & 0.21 & 80  & 0.21025 & 0.0012 & 80.179 & 0.0022 \\
    \end{tabular}
    \caption{Attenuation and emission values of one rod reconstructed using the asymptotic limits of $\Phi_D$ and $\Phi_N$ at a perfectly tangential ray.}
    \label{tab:recons}
\end{table}

\subsection*{Acknowledgments}
The work of T.~Heikkil\"a, T.~Helin, S.R.J., and J.R. was supported by the Research Council of Finland through the Flagship of Advanced Mathematics for Sensing, Imaging and Modelling (decision number 359183). In addition the work of T.~Heikkil\"a and T.~Helin was supported by the Business Finland TomoHQ project. The work of S.R.J. and J.R. was also supported by the Jenny and Antti Wihuri Foundation. The work of S.H. was supported by the Engineering and Physical Sciences Research Council (EPSRC) through grant EP/V007742/1. S.R.J. acknowledges the support of the DKO Trust and the European Mathematical Society for his visit to the University of Manchester,  where part of this work was completed, and thanks the University of Manchester for its hospitality. In addition, J.R. was supported by the Emil Aaltonen Foundation.

\subsection*{Data availability statement}
 The Python code for our numerical examples with instructions is available at \url{github.com/tommheik/PGET_stability}.

\subsection*{Conflict of interest}
The authors declare no potential conflicts of interest with respect to the research, authorship, or publication of this article.

\subsection*{Declaration of generative AI use}
During the preparation of this work, the authors used Claude Opus 5/5.5 and Fable 5.1 to help simplify certain calculations, to improve the presentation and to help implement the numerical experiments. The authors reviewed and verified the AI-assisted content.
 \bibliography{references} 

@article{Kunyansky:2001,
  author  = {Kunyansky, Leonid A.},
  title   = {A new {SPECT} reconstruction algorithm based
             on the {Novikov} explicit inversion formula},
  journal = {Inverse Problems},
  volume  = {17},
  number  = {2},
  pages   = {293--306},
  year    = {2001}
}

@incollection {Finch:2003,
    AUTHOR = {Finch, David V.},
     TITLE = {The attenuated x-ray transform: recent developments},
 BOOKTITLE = {Inside out: inverse problems and applications},
    SERIES = {Math. Sci. Res. Inst. Publ.},
    VOLUME = {47},
     PAGES = {47--66},
 PUBLISHER = {Cambridge Univ. Press, Cambridge},
      YEAR = {2003},
      ISBN = {0-521-82469-9},
   MRCLASS = {44A12 (30E20 92C55)},
  MRNUMBER = {2029677},
MRREVIEWER = {Eric\ Todd\ Quinto},
}

@article {Boman:Stromberg:2004,
    AUTHOR = {Boman, Jan and Str\"omberg, Jan-Olov},
     TITLE = {Novikov's inversion formula for the attenuated {R}adon
              transform---a new approach},
   JOURNAL = {J. Geom. Anal.},
  FJOURNAL = {The Journal of Geometric Analysis},
    VOLUME = {14},
      YEAR = {2004},
    NUMBER = {2},
     PAGES = {185--198},
      ISSN = {1050-6926,1559-002X},
   MRCLASS = {44A12 (47G10 92C55)},
  MRNUMBER = {2051682},
MRREVIEWER = {Fritz\ Keinert},
       DOI = {10.1007/BF02922067},
       URL = {https://doi-org.ezproxy.cc.lut.fi/10.1007/BF02922067},
}

@article {Natterer:2001:inversion,
    AUTHOR = {Natterer, F.},
     TITLE = {Inversion of the attenuated {R}adon transform},
   JOURNAL = {Inverse Problems},
  FJOURNAL = {Inverse Problems. An International Journal on the Theory and
              Practice of Inverse Problems, Inverse Methods and Computerized
              Inversion of Data},
    VOLUME = {17},
      YEAR = {2001},
    NUMBER = {1},
     PAGES = {113--119},
      ISSN = {0266-5611,1361-6420},
   MRCLASS = {44A12},
  MRNUMBER = {1818495},
       DOI = {10.1088/0266-5611/17/1/309},
       URL = {https://doi-org.ezproxy.cc.lut.fi/10.1088/0266-5611/17/1/309},
}

@article {Hertle:1988,
    AUTHOR = {Hertle, Alexander},
     TITLE = {The identification problem for the constantly attenuated
              {R}adon transform},
   JOURNAL = {Math. Z.},
  FJOURNAL = {Mathematische Zeitschrift},
    VOLUME = {197},
      YEAR = {1988},
    NUMBER = {1},
     PAGES = {13--19},
      ISSN = {0025-5874,1432-1823},
   MRCLASS = {44A15 (47B38 92A07)},
  MRNUMBER = {917847},
MRREVIEWER = {Frank\ Natterer},
       DOI = {10.1007/BF01161627},
       URL = {https://doi-org.ezproxy.cc.lut.fi/10.1007/BF01161627},
}

@article {Holman:Richardson:2020,
    AUTHOR = {Holman, Sean and Richardson, Philip},
     TITLE = {S{PECT} with a multi-bang assumption on attenuation},
   JOURNAL = {Inverse Problems},
  FJOURNAL = {Inverse Problems. An International Journal on the Theory and
              Practice of Inverse Problems, Inverse Methods and Computerized
              Inversion of Data},
    VOLUME = {36},
      YEAR = {2020},
    NUMBER = {12},
     PAGES = {125005, 31},
      ISSN = {0266-5611,1361-6420},
   MRCLASS = {65N21 (35R30 44A12)},
  MRNUMBER = {4186175},
MRREVIEWER = {Thomas\ Schuster},
       DOI = {10.1088/1361-6420/abab59},
       URL = {https://doi-org.ezproxy.cc.lut.fi/10.1088/1361-6420/abab59},
}

@article {Tretiak:1980,
    AUTHOR = {Tretiak, Oleh and Metz, Charles},
     TITLE = {The exponential {R}adon transform},
   JOURNAL = {SIAM J. Appl. Math.},
  FJOURNAL = {SIAM Journal on Applied Mathematics},
    VOLUME = {39},
      YEAR = {1980},
    NUMBER = {2},
     PAGES = {341--354},
      ISSN = {0036-1399},
   MRCLASS = {44A15},
  MRNUMBER = {588505},
MRREVIEWER = {L.\ Debnath},
       DOI = {10.1137/0139029},
       URL = {https://doi-org.ezproxy.cc.lut.fi/10.1137/0139029},
}

@article {Solmon:1995,
    AUTHOR = {Solmon, Donald C.},
     TITLE = {The identification problem for the exponential {R}adon
              transform},
   JOURNAL = {Math. Methods Appl. Sci.},
  FJOURNAL = {Mathematical Methods in the Applied Sciences},
    VOLUME = {18},
      YEAR = {1995},
    NUMBER = {9},
     PAGES = {687--695},
      ISSN = {0170-4214,1099-1476},
   MRCLASS = {44A12},
  MRNUMBER = {1339930},
       DOI = {10.1002/mma.1670180903},
       URL = {https://doi-org.ezproxy.cc.lut.fi/10.1002/mma.1670180903},
}

@article {Arbuzov:Bukhegim:Kazantsev:1998,
    AUTHOR = {Arbuzov, {\`E}. V. and Bukhge{\u{\i}}m, A. L. and Kazantsev, S. G.},
     TITLE = {Two-dimensional tomography problems and the theory of {$A$}-analytic functions},
   JOURNAL = {Siberian Adv. Math.},
  FJOURNAL = {Siberian Advances in Mathematics},
    VOLUME = {8},
      YEAR = {1998},
    NUMBER = {4},
     PAGES = {1--20},
      ISSN = {1055-1344,1934-8126},
   MRCLASS = {44A12 (30G20)},
  MRNUMBER = {1666500},
MRREVIEWER = {Faizo\ A.\ Shamoyan},
}

@book {Natterer:2001:book,
    AUTHOR = {Natterer, F.},
     TITLE = {The mathematics of computerized tomography},
 PUBLISHER = {B. G. Teubner, Stuttgart; John Wiley \& Sons, Ltd.,
              Chichester},
      YEAR = {1986},
     PAGES = {x+222},
      ISBN = {3-519-02103-X},
   MRCLASS = {44A15 (92A07)},
  MRNUMBER = {856916},
MRREVIEWER = {L.\ A.\ Shepp},
}

@article {Bubba:2020:pget,
    AUTHOR = {Backholm, Rasmus and Bubba, Tatiana A. and
              B\'elanger-Champagne, Camille and Helin, Tapio and Dendooven,
              Peter and Siltanen, Samuli},
     TITLE = {Simultaneous reconstruction of emission and attenuation in
              passive gamma emission tomography of spent nuclear fuel},
   JOURNAL = {Inverse Probl. Imaging},
  FJOURNAL = {Inverse Problems and Imaging},
    VOLUME = {14},
      YEAR = {2020},
    NUMBER = {2},
     PAGES = {317--337},
      ISSN = {1930-8337,1930-8345},
   MRCLASS = {45Q05 (65R32)},
  MRNUMBER = {4155068},
       DOI = {10.3934/ipi.2020014},
       URL = {https://doi-org.ezproxy.cc.lut.fi/10.3934/ipi.2020014},
}

@book {primer:book,
    AUTHOR = {Krantz, Steven G. and Parks, Harold R.},
     TITLE = {A primer of real analytic functions},
    SERIES = {Birkh{\"a}user Advanced Texts: Basel Textbooks},
   EDITION = {Second},
 PUBLISHER = {Birkh\"auser Boston, Inc., Boston, MA},
      YEAR = {2002},
     PAGES = {xiv+205},
      ISBN = {0-8176-4264-1},
   MRCLASS = {26E05 (26-02 26E10 32B20)},
  MRNUMBER = {1916029},
MRREVIEWER = {Solomon\ Marcus},
       DOI = {10.1007/978-0-8176-8134-0},
       URL = {https://doi-org.ezproxy.cc.lut.fi/10.1007/978-0-8176-8134-0},
}

@phdthesis{Virta:2024,  
author       = {Virta, Riina},  
title        = {Gamma tomography of spent nuclear fuel for geological repository safeguards},  
school       = {University of Helsinki},  
address      = {Helsinki, Finland},  
year         = {2024},  
type         = {PhD Thesis, HIP Internal Report, HIP-2024-02},  
pages        = {60+40},  
isbn         = {978-951-51-8427-6},  
issn         = {1455-0563},  
url          = {http://urn.fi/URN:ISBN:978-951-51-8428-3},  
note         = {ISBN (PDF): 978-951-51-8428-3; ISSN (online): 2814-9459}
}

@article {Stefanov:2014,
    AUTHOR = {Stefanov, Plamen},
     TITLE = {The identification problem for the attenuated {X}-ray
              transform},
   JOURNAL = {Amer. J. Math.},
  FJOURNAL = {American Journal of Mathematics},
    VOLUME = {136},
      YEAR = {2014},
    NUMBER = {5},
     PAGES = {1215--1247},
      ISSN = {0002-9327,1080-6377},
   MRCLASS = {44A12},
  MRNUMBER = {3263897},
MRREVIEWER = {B.\ S.\ Rubin},
       DOI = {10.1353/ajm.2014.0035},
       URL = {https://doi-org.ezproxy.cc.lut.fi/10.1353/ajm.2014.0035},
}

@article {Holman:Richardson:2023,
    AUTHOR = {Holman, Sean and Richardson, Philip},
     TITLE = {Simultaneous recovery of attenuation and source density in
              {SPECT}},
   JOURNAL = {Inverse Probl. Imaging},
  FJOURNAL = {Inverse Problems and Imaging},
    VOLUME = {17},
      YEAR = {2023},
    NUMBER = {4},
     PAGES = {817--840},
      ISSN = {1930-8337,1930-8345},
   MRCLASS = {45G05 (44A12 65R10 92C55)},
  MRNUMBER = {4576183},
MRREVIEWER = {Fritz\ Keinert},
       DOI = {10.3934/ipi.2023005},
       URL = {https://doi.org/10.3934/ipi.2023005},
}

@article {Novikov:2002,
    AUTHOR = {Novikov, Roman G.},
     TITLE = {An inversion formula for the attenuated {X}-ray
              transformation},
   JOURNAL = {Ark. Mat.},
  FJOURNAL = {Arkiv f\"or Matematik},
    VOLUME = {40},
      YEAR = {2002},
    NUMBER = {1},
     PAGES = {145--167},
      ISSN = {0004-2080,1871-2487},
   MRCLASS = {44A12 (35R30 65R10 82C70 92C55)},
  MRNUMBER = {1948891},
MRREVIEWER = {Olga\ Klimenko},
       DOI = {10.1007/BF02384507},
       URL = {https://doi.org/10.1007/BF02384507},
}

@article {Novikov:2002:range,
    AUTHOR = {Novikov, R. G.},
     TITLE = {On the range characterization for the two-dimensional
              attenuated {X}-ray transformation},
   JOURNAL = {Inverse Problems},
    VOLUME = {18},
      YEAR = {2002},
    NUMBER = {3},
     PAGES = {677--700},
      ISSN = {0266-5611,1361-6420},
       DOI = {10.1088/0266-5611/18/3/310},
       URL = {https://doi.org/10.1088/0266-5611/18/3/310},
}

@incollection {Zhou:2024:survey,
    AUTHOR = {Zhou, Hanming},
     TITLE = {The weighted {X}-ray transform and applications},
 BOOKTITLE = {Microlocal analysis and inverse problems in tomography and
              geometry},
    SERIES = {Radon Ser. Comput. Appl. Math.},
    VOLUME = {30},
     PAGES = {167--185},
 PUBLISHER = {De Gruyter, Berlin},
      YEAR = {2024},
      ISBN = {978-3-11-133683-1; 978-3-11-133748-7; 978-3-11-133801-9},
   MRCLASS = {44-02 (35R30 35S05 44A12 53C24 53C65)},
  MRNUMBER = {4815455},
}

@incollection {Novikov:2019,
    AUTHOR = {Novikov, R. G.},
     TITLE = {Non-{A}belian {R}adon transform and its applications},
 BOOKTITLE = {The {R}adon transform---the first 100 years and beyond},
    SERIES = {Radon Ser. Comput. Appl. Math.},
    VOLUME = {22},
     PAGES = {115--127},
 PUBLISHER = {Walter de Gruyter, Berlin},
      YEAR = {2019},
      ISBN = {978-3-11-056085-5; 978-3-11-055951-4; 978-3-11-055941-5},
   MRCLASS = {44A12 (35F05 35P25 35R30 37K15 53C65 78A46 81U40)},
  MRNUMBER = {4059257},
       DOI = {10.1515/9783110560855-005},
       URL = {https://doi.org/10.1515/9783110560855-005},
}

@book {GIP2D,
    AUTHOR = {Paternain, Gabriel P. and Salo, Mikko and Uhlmann, Gunther},
     TITLE = {Geometric inverse problems---with emphasis on two dimensions},
    SERIES = {Cambridge Studies in Advanced Mathematics},
    VOLUME = {204},
      NOTE = {With a foreword by Andr\'{a}s Vasy},
 PUBLISHER = {Cambridge University Press, Cambridge},
      YEAR = {2023},
     PAGES = {xxiv+344},
      ISBN = {978-1-316-51087-2},
   MRCLASS = {35-02 (35R30 53C65 58J32 92C55)},
  MRNUMBER = {4520155},
}

@article{censor1979new,
  title={A new approach to the emission computerized tomography problem: simultaneous calculation of attenuation and activity coefficients},
  author={Censor, Yair and Gustafson, David E. and Lent, Arnold and Tuy, Heang},
  journal={{IEEE} Transactions on Nuclear Science},
  volume={26},
  number={2},
  pages={2775--2779},
  year={1979},
  publisher={IEEE}
}

@article{shepp1982maximum,
  title={Maximum likelihood reconstruction for emission tomography},
  author={Shepp, Lawrence A. and Vardi, Yehuda},
  journal={{IEEE} transactions on medical imaging},
  volume={1},
  number={2},
  pages={113--122},
  year={1982},
  publisher={IEEE}
}

@article{hudson1994accelerated,
  title={Accelerated image reconstruction using ordered subsets of projection data},
  author={Hudson, H. Malcolm and Larkin, Richard S.},
  journal={{IEEE} transactions on medical imaging},
  volume={13},
  number={4},
  pages={601--609},
  year={1994},
  publisher={IEEE}
}

@article{natterer1983computerized,
  title={Computerized tomography with unknown sources},
  author={Natterer, Frank},
  journal={{SIAM} J. Appl. Math.},
  volume={43},
  number={5},
  pages={1201--1212},
  year={1983},
  publisher={SIAM}
}

@article{bronnikov1999numerical,
  title={Numerical solution of the identification problem for the attenuated {Radon} transform},
  author={Bronnikov, Andrei V.},
  journal={Inverse Problems},
  volume={15},
  number={5},
  pages={1315--1324},
  year={1999}
}

@inproceedings{white2018application,
  title={Application of passive gamma emission tomography ({PGET}) for the verification of spent nuclear fuel},
  author={White, Timothy and Mayorov, Mikhail and Lebrun, Alain and Peura, Pauli and Honkamaa, Tapani and Dahlberg, Joakim and Keubler, Jens and Ivanov, Victor and Turunen, Asko},
  booktitle={INMM 59th Annual Meeting, Baltimore, Maryland, USA},
  year={2018}
}

@article{virta2022improved,
  title={Improved Passive Gamma Emission Tomography image quality in the central region of spent nuclear fuel},
  author={Virta, Riina and Bubba, Tatiana A. and Moring, Mikael and Siltanen, Samuli and Honkamaa, Tapani and Dendooven, Peter},
  journal={Scientific Reports},
  volume={12},
  number={1},
  pages={12473},
  year={2022},
  publisher={Nature Publishing Group UK London}
}

@article{dicken1999new,
  title={A new approach towards simultaneous activity and attenuation reconstruction in emission tomography},
  author={Dicken, Volker},
  journal={Inverse Problems},
  volume={15},
  number={4},
  pages={931--960},
  year={1999}
}

@article{svard2015applicability,
  title={Applicability of a set of tomographic reconstruction algorithms for quantitative {SPECT} on irradiated nuclear fuel assemblies},
  author={Sv{\"a}rd, Staffan Jacobsson and Holcombe, Scott and Grape, Sophie},
  journal={Nuclear Instruments and Methods in Physics Research Section A: Accelerators, Spectrometers, Detectors and Associated Equipment},
  volume={783},
  pages={128--141},
  year={2015},
  publisher={Elsevier}
}

@article{fang2021quantitative,
  title={Quantitative imaging and automated fuel pin identification for passive gamma emission tomography},
  author={Fang, Ming and Altmann, Yoann and Della Latta, Daniele and Salvatori, Massimiliano and Di Fulvio, Angela},
  journal={Scientific reports},
  volume={11},
  number={1},
  pages={2442},
  year={2021},
  publisher={Nature Publishing Group UK London}
}

@inproceedings{white2017spect,
  title={{SPECT} reconstruction and analysis for the inspection of spent nuclear fuel},
  author={White, Timothy and Mayorov, Mikhail and Deshmukh, Nikhil and Miller, Erin and Smith, L. Eric and Dahlberg, Joakim and Honkamaa, Tapani},
  booktitle={2017 IEEE Nuclear Science Symposium and Medical Imaging Conference (NSS/MIC)},
  pages={1--2},
  year={2017},
  organization={IEEE}
}
	
	\bibliographystyle{alpha}

\end{document}